\documentclass[a4paper, 11pt, oneside]{amsart}

\title{Canonical Metrics on Families of Vector Bundles}
\date{\today}
\author{Shing Tak Lam}
\address{Mathematics Institute, University of Warwick, Coventry CV4 7AL, United Kingdom}
\email{Shing-Tak.Lam@warwick.ac.uk}

\usepackage[margin=1in]{geometry}

\usepackage[UKenglish]{babel}
\usepackage{csquotes}

\usepackage[ocgcolorlinks, linkcolor=blue, citecolor=cyan]{hyperref}

\usepackage{tikz-cd}

\usepackage{amssymb}
\usepackage{physics}
\usepackage{mathtools}
\usepackage{enumerate}
\usepackage[shortlabels]{enumitem}
\usepackage{caption}
\usepackage{hypcap}

\renewcommand{\grad}{\nabla}
\renewcommand{\epsilon}{\varepsilon}

\DeclareMathOperator{\id}{id}
\DeclareMathOperator{\End}{End}
\DeclareMathOperator{\Aut}{Aut}
\DeclareMathOperator{\GL}{GL}
\DeclareMathOperator{\pr}{pr}

\DeclareMathOperator{\vol}{vol}

\DeclareMathOperator{\Int}{Int}
\DeclareMathOperator{\Lie}{Lie}

\newcommand{\Z}{\mathbb Z}
\newcommand{\R}{\mathbb R}
\newcommand{\C}{\mathbb C}
\newcommand{\delbar}{\bar\partial}
\newcommand{\del}{\partial}

\newcommand{\mcA}{\mathcal A}
\newcommand{\mcD}{\mathcal D}
\newcommand{\mcE}{\mathcal E}
\newcommand{\mcF}{\mathcal F}

\newcommand{\mcK}{\mathcal K}
\newcommand{\mcL}{\mathcal L}
\newcommand{\mcM}{\mathcal M}
\newcommand{\mcN}{\mathcal N}

\newcommand{\mcP}{\mathcal P}
\newcommand{\mcQ}{\mathcal Q}
\newcommand{\mcR}{\mathcal R}

\usepackage{mathrsfs}
\newcommand{\msH}{\mathscr H}

\newcommand{\mfg}{\mathfrak g}
\newcommand{\mfk}{\mathfrak k}

\usepackage{amsthm}
\usepackage{thmtools}
\usepackage[noabbrev]{cleveref}

\newtheorem{theorem}{Theorem}[section]
\newtheorem{proposition}[theorem]{Proposition}
\newtheorem{lemma}[theorem]{Lemma}
\newtheorem{corollary}[theorem]{Corollary}

\theoremstyle{definition}
\newtheorem{definition}[theorem]{Definition}
\newtheorem{example}[theorem]{Example}
\newtheorem{remark}[theorem]{Remark}
\newtheorem{assumption}[theorem]{Assumption}

\numberwithin{equation}{section}

\begin{document}

\begin{abstract}
    We introduce a geometric partial differential equation for families of holomorphic vector bundles, generalising the theory of Hermite--Einstein metrics. We consider families of holomorphic vector bundles which each admit Hermite--Einstein metrics, together with a first order deformation. On such families, we define the \emph{family Hermite--Einstein equation} for Hermitian metrics, which we view as a notion of a canonical metric in this setting.

    We prove two main results concerning family Hermite--Einstein metrics. Firstly, we construct Hermite--Einstein metrics in adiabatic classes on product manifolds, assuming the existence of a family Hermite--Einstein metric. Secondly, we prove that the associated parabolic flow admits a unique smooth solution for all time, and use this to show that the Dirichlet problem always admits a unique solution.
\end{abstract}

\maketitle

\section{Introduction}

In geometry, one of the central problems is the study of \emph{canonical metrics}, both on manifolds and vector bundles. Let \((X, \omega_X)\) be a compact K\"ahler manifold, \(\mcE \to X\) a holomorphic vector bundle. A Hermitian metric \(h\) on \(\mcE\) is \emph{Hermite--Einstein} if
\[i\Lambda_{\omega_X}F_h = c\id\]
where \(c\) is a topological constant. The Hitchin--Kobayashi correspondence of Donaldson and Uhlenbeck--Yau \cite{donaldsonSelfdualYangMillsConnections1985,uhlenbeckExistenceHermitianYangMillsConnections1986} relates existence of Hermite--Einstein metrics to the slope polystability of the bundle, which is an algebro-geometric notion, andwhich  leads to a moduli theory for holomorphic vector bundles. Finding analogues of the Hitchin--Kobayashi correspondence in different geometric scenarios is a topic of active research.

In this paper, we initiate the study of canonical metrics on a \emph{family} of holomorphic vector bundles, by defining a new partial differential equation, called the \emph{family Hermite--Einstein} equation.

\subsection{Family Hermite--Einstein metrics}

A \emph{family of holomorphic vector bundles} on \((X, \omega_X)\) parametrised by \((B, \omega_B)\) is a holomorphic vector bundle \(\mcE = (E, \delbar) \to X \times B\). We will primarily be interested in the case when \(B\) is compact, with or without boundary.

A Hermitian metric on \(E\) induces Hermitian metrics on the restricted bundles \(\mcE_b = \mcE\vert_{X \times \{b\}} \to X\), which we assume are Hermite--Einstein. We include in addition a deformation \(\delbar_s\) of the Dolbeault operator of \(\mcE\), from which we obtain an expansion of the form
\[p(i\Lambda_V F_{h, \delbar_s}) = -s^2i\nu(a) + \order{s^3},\]
where
\[a = \pdv{s}\bigg\vert_{s=0}\alpha_V(s)\]
is the first order vertical deformation, \(p\) is a projection operator onto the space of fibrewise holomorphic endomorphisms of \(\mcE\), \(\Lambda_V\) is the contraction operator with respect to \(\omega_X\), and \(\nu\) is a moment map. We say that the Hermitian metric \(h\) satisfies the \emph{family Hermite--Einstein} equation if
\[p(i\Lambda_H F_h) - i\lambda\nu(a) = 0.\]
Here, \(\Lambda_H\) is the contraction operator with respect to \(\omega_B\), and \(\lambda \in \R_{>0}\) is a coupling constant.

We will assume that the dimensions of the automorphism groups \(G_b = \Aut(\mcE_b)\) are equal. Under this assumption, we obtain a holomorphic vector bundle \(\mcF \to B\), whose fibre over \(b \in B\) is \(H^0(X, \End(\mcE_b))\).
The map \(p\) is then defined for each \(b \in B\) to be the trace-free part of the \(L^2\)-orthogonal projection \(\mcA^0(X, \End(E_b)) \to H^0(X, \End(\mcE_b))\), and we view the family Hermite--Einstein equation as a partial differential equation on sections of \(\mcF\). We note that if each \(\mcE_b\) is stable, then the projection operator is zero, and hence the family Hermite--Einstein equation is vacuous. Thus, the family Hermite--Einstein equation enters precisely because of the non-uniqueness of Hermite--Einstein metrics in the presence of automorphisms.

\subsection{Motivation}

There are many reasons to be interested in studying the family Hermite--Einstein equation. We briefly outline some of them here, and will discuss them in detail in \Cref{subsec:relation-to-prior-work}.

\begin{enumerate}[(i)]
    \item From moduli theory, we may construct a moduli space or stack \(\mcM\) of semistable bundles over \(X\). A morphism \(B \to \mcM\) corresponds to a family of semistable bundles, each of which has a degeneration to a polystable bundle. The geometry of families of bundles is then be related to the geometry of \(\mcM\), and we view the family Hermite--Einstein equation as producing the correct notion of a canonical metric on families of semistable bundles, noting however that this depends only on the first order deformation from the associated polystable family. Note also that the Hermite--Einstein theory relates to maps from \(B\) to the quotient stack \([\mathrm{pt}/\mathrm{GL}]\), and in our setting, we replace the target stack with the moduli stack of semistable bundles over \(X\).
    \item In gauge theory, the problem of constructing Yang--Mills connections in adiabatic classes has a long history, such as in the work of Dostoglou--Salamon \cite{dostoglouSelfdualInstantonsHolomorphic1994} in relation to the Atiyah--Floer conjecture. Roughly, in this work, families of irreducible bundles play an important role, whereas in our work, for families of semistable or polystable bundles, we encounter a new partial differential equation as an obstruction to the analogous construction.
    \item Conversely, the geometry of (anti-self-dual) Yang--Mills connections on \((X \times B, \omega_X + k\omega_B)\) as \(k \to \infty\) has been shown to be related to maps from \(B\) to an appropriate moduli space. We conjecture that the family Hermite--Einstein equation determines the geometry of such collapsing families, away from singularities.
    \item Finally, a similar question may be posed for families of projective varieties, and this is the theory of \emph{optimal symplectic connections} \cite{dervanOptimalSymplecticConnections2021,ortuOptimalSymplecticConnections2023}. The analytic theory for the family Hermite--Einstein equations is more tractable than for optimal symplectic connections, much as the analytic theory for Hermite--Einstein metrics is more tractable than for constant scalar curvature K\"ahler metrics. This allows us to develop the analytic theory substatially further than the optimal symplectic connection theory, as we shall see.
\end{enumerate}

\subsection{Results}

Our first result involves construction of Hermite--Einstein metrics in adiabatic classes. Let \(\omega_k = \omega_X + k\omega_B\). We are interested in the problem of constructing Hermitian metrics \(h_k\) for \(k \gg 0\), such that
\[i\Lambda_{\omega_k}F_{h_k} = c_k\id.\]
Including a deformation \(\delbar_s\), we consider the problem of constructing Hermite--Einstein metrics on \((E, \delbar_s)\), where we allow \((E, \delbar_s)\) and \((E, \delbar_{s'})\) to be non-isomorphic for \(s \ne s'\). We will relate the parameters \(k\) and \(s\), by requiring \(s^2 = \lambda k^{-1}\).

\begin{theorem}
    [\Cref{thm:hym-total-space}]
    \label{thm:intro-adiabatic-he}
    Let \((X, \omega_X)\) and \((B, \omega_B)\) be compact K\"ahler manifolds. Let \(\mcE = (E, \delbar) \to X \times B\) be a holomorphic vector bundle, \(h\) a Hermitian metric on \(E\). Suppose that for all \(b \in B\), \((\mcE_b, h) \to X\) is Hermite--Einstein. Let \(\delbar_s\) be a deformation of \(\mcE\), with \(a\) the first order deformation corresponding to \(\delbar_s\).

    Suppose \(h\) satisfies the family Hermite--Einstein equation, and that \(\Aut(\mcE, a_V) = \C^*\cdot \id\). Then for \(k \gg 0\), there exists a Hermitian metric \(h_k\) for \(E\), such that
    \[i\Lambda_{\omega_k}F_{h_k, \delbar_s} = c_k\id.\]
\end{theorem}

Thus, the family Hermite--Einstein equation arises as the crucial analytic obstruction to finding Hermite--Einstein metrics in adiabatic classes, and controls the geometry of the resulting Hermite--Einstein metrics, in line with (iii) above. We note that if each \(\mcE_b\) is stable, then the family Hermite--Einstein equation is trivially satisfied, and \Cref{thm:hym-total-space} gives a new construction of Hermite--Einstein metrics on \((\mcE, \delbar_s)\). Our primary interest, however, is in the case when the \(\mcE_b\) are semistable or polystable, which is when the family Hermite--Einstein eqution becomes relevant.

The construction of the Hermitian metrics will follow the general strategy as developed in \cite{fineConstantScalarCurvature2004a,hongConstantHermitianScalar1999,dervanOptimalSymplecticConnections2021,ortuOptimalSymplecticConnections2023} for constant scalar curvature K\"ahler (cscK) metrics, and \cite{sektnanHermitianYangMillsConnections2024} for Hermitian--Yang--Mills connections. A key point is to develop the linearised theory of the family Hermite--Einstein equation, which we may consider through the linearised operator \(\mcL \colon \mcA^0(\mcF) \to \mcA^0(\mcF)\), and which we will prove that this is a self-adjoint second-order elliptic operator. Our assumption
\[\Aut(\mcE, a_V) = \{\sigma \in \Aut(\mcE) \mid \sigma a_V = a_V\sigma\} = \C^*\cdot \id,\]
will imply that \(\mcL\) is invertible. We expect that this assumption may be removed in general, at the expense of a more technical statement.

Next, an important tool in the study of Hermite--Einstein metrics is the associated parabolic flow
\[h^{-1}\pdv{h}{t} = -2(i\Lambda_\omega F_h - c\id).\]
In \cite{donaldsonSelfdualYangMillsConnections1985}, Donaldson showed that when \((X, \omega_X)\) is closed, the flow exists for all time. This was generalised by Simpson \cite{simpsonConstructingVariationsHodge1988} to the case when \(X\) has boundary.

We prove an analogous statement in our setting:

\begin{theorem}
    [\Cref{thm:hhe-flow-exists-for-all-time}]
    \label{thm:intro-hhe-flow}
    Suppose \((X, \omega_X)\) is compact K\"ahler, \((B, \omega_B)\) compact K\"ahler with or without boundary. Let \(\mcE \to X\) be a holomorphic vector bundle, \(\delbar_s\) a deformation of the Dolbeault operator of \(\mcE\). When \(B\) has boundary, impose Dirichlet boundary conditions with \(h = h_\partial\) on \(X \times \partial B\), where \(h_\partial\) is a fixed vertically Hermite--Einstein metric on \(\mcE\vert_{X \times \partial B}\).

    The family Hermite--Einstein flow
    \[h^{-1}\pdv{h}{t} = -2\left(p(i\Lambda_H F_h) - \lambda i\nu_h(a)\right)\]
    has a unique smooth solution for all time.
\end{theorem}

The statement is true for \emph{any} \(\lambda > 0\). A key point in our work is to manage the deformation term \(\nu(a)\) along the flow, for which we employ moment map geometry.

Finally, in \cite{donaldsonBoundaryValueProblems1992}, Donaldson uses Simpson's result and the maximum principle to show that the Dirichlet boundary value problem of the Hermite--Einstein equations always has a solution. Using \Cref{thm:intro-hhe-flow} and Donaldson's approach, we prove:

\begin{theorem}
    [\Cref{thm:dirichlet-problem}]
    \label{thm:intro-dirichlet-problem}
    Suppose \(\partial B \ne \emptyset\). Then the Dirichlet problem for the family Hermite--Einstein equation
    \[\begin{cases}
        p_h(i\Lambda_H F_h) - i\lambda \nu_h(a) = 0 & \text{on } X \times \Int(B) \\
        h = h_\partial & \text{on } X \times \partial B,
    \end{cases}\]
    admits a unique solution, for a fixed vertically Hermite--Einstein metric \(h_\partial\) on the smooth vector bundle \(E\vert_{X \times\partial B}\).
\end{theorem}

This produces many non-trivial solutions of the family Hermite--Einstein equations, provided \(B\) has boundary.

\subsection{Relation to prior work}

\label{subsec:relation-to-prior-work}

Motivation for our work comes from the corresponding question for families of polarised varieties. The notion of a canonical metric on a polarised variety is given by \emph{constant scalar curvature K\"ahler} metrics. Given a holomorphic submersion \(X \to B\), let \(\omega_X\) be a relatively K\"ahler metric, which has constant scalar curvature on each fibre. A natural question is then: what is the canonical choice of \(\omega_X\)? Generalising work of Fine \cite{fineConstantScalarCurvature2004a} and Hong \cite{hongConstantHermitianScalar1999} constructing cscK metrics in adiabatic classes, Dervan--Sektnan introduced the notion of an \emph{optimal symplectic connection} \cite{dervanOptimalSymplecticConnections2021}, which they proposed as an answer to this question. This was then generalised by Ortu \cite{ortuOptimalSymplecticConnections2023} to include a first order deformation term, analogously to what is carried out in this paper, which is more natural from the perspective of moduli theory.

The analysis involved in the vector bundle setting is generally more tractable than in the setting of polarised varieties. Our expectation is thus that the analytic theory behind the family Hermite--Einstein equations should be more tractable than for optimal symplectic connections. As a first step, we prove the long-time existence of the associated parabolic flow, while there are no general analytic results for optimal symplectic connections. We note that there is a discrepancy between the linear theory as developed in \cite{ortuOptimalSymplecticConnections2023} and in this paper, which leads to a slightly different geometric interpretation. We refer to \Cref{rem:difference-with-ortu} for details.

When \(X = \mathbb P(\mcE)\) is the projectivisation of a holomorphic vector bundle \(\mcE \to B\), the optimal symplectic connection equation reduces to the Hermite--Einstein equation \cite[Section 3.5]{dervanOptimalSymplecticConnections2021}. Another equation on the total space of a fibration which which reduces to the Hermite--Einstein equation for projective bundles is the \emph{Wess--Zumino--Witten} equation. We refer to \cite{finskiWessZuminoWittenEquationHarderNarasimhan2024,finskiLowerBoundsFibered2024} and references therein for details. A key difference is that the Wess--Zumino--Witten equation does not involve a projection operator, and thus the linearised operator is not elliptic. Many of the results and objects in this paper have analogues in the Wess--Zumino--Witten theory. For example, \cite{finskiWessZuminoWittenEquationHarderNarasimhan2024} describes the obstructions to the existence of (approximate) solutions, and \cite{wuPotentialTheoryWess2024} proves the solvability of the Dirichlet problem, 

On the gauge theory side, constructing Yang--Mills connections using adiabatic limits is a well studied problem. In \cite{dostoglouSelfdualInstantonsHolomorphic1994}, Dostoglou--Salamon used adiabatic limits in relation to the Atiyah--Floer conjecture, and in \cite{sektnanHermitianYangMillsConnections2024}, Sektnan--Tipler studied the problem of constructing Hermite--Einstein metrics on the pullback of a bundle in adiabatic classes.
Conversely, an important question is the limiting behaviour of Yang--Mills connections in adiabatic classes. In \cite{hongHarmonicMapsModuli1999}, Hong considers the problem of Yang--Mills connections on \(X \times B\), where \(X\) is a Riemann surface of genus at least two. Hong shows that harmonic maps \(B \to \mcM_X\) to the moduli space of flat connections are given by adiabatic limits of Yang--Mills connections on \(X \times B\). Relatedly, \cite{chenConvergenceAntiselfdualConnections1998,chenComplexAntiselfdualConnections1999,datarAdiabaticLimitsAntiselfdual2021,datarHermitianYangMillsConnectionsCollapsing2022} considers the convergence of Yang--Mills connections in the adiabatic limit, and studies the \emph{bubbling} behaviour. 

In contrast to prior work, we do not assume that the bundles involved are irreducible, which leads to our new geometric PDE. We conjecture that over an open subset away from the bubbling loci, the limiting connection on \(X \times B\) should solve the family Hermite--Einstein equation. This is analogous to the geometry of collapsing Calabi--Yau manifolds, where the relevant fibres are generally Calabi--Yau manifolds, hence have a \emph{unique} Calabi--Yau metric. By contrast, the family Hermite--Einstein equation enters in the analogous bundle theory, precisely because of the lack of uniqueness of Hermite--Einstein metrics in the presence of automorphisms.

Turning to moduli theory, the geometry in the optimal symplectic connection and family Hermite--Einstein problems are given by the corresponding moduli functors. In both cases, we have a map \(B \to \mcM\), where \(\mcM\) is the appropriate moduli space or stack. The moduli spaces are given by symplectic reduction, and the study of maps to symplectic quotients is related to \emph{gauged sigma models}, \emph{principal pairs} and \emph{symplectic vortices} \cite{gaioJholomorphicCurvesMoment1999,banfieldStablePairsPrincipal2000,mundetirieraHitchinKobayashiCorrespondenceKahler2000,cieliebakJholomorphicCurvesMoment2000,cieliebakSymplecticVortexEquations2002,bradlowRelativeHitchinKobayashiCorrespondences2003}. These are equations of the form
\[\Lambda_\omega F_A + \mu(\Phi) = 0\]
where \(F_A\) is the curvature of a principal \(K\)-bundle \(P \to X\), \(E \to X\) an associated bundle with fibre \(F\), \(\mu \colon F \to \mfk^*\) a moment map for the \(K\)-action on \(F\), and \(\Phi\) is a section of \(E\).

We conjecture that there should be a way to view the family Hermite--Einstein equation as an infinite dimensional version of the vortex equations, for example as considered in \cite{cieliebakJholomorphicCurvesMoment2000} for the anti-self-dual Yang--Mills and Seiberg--Witten equations. Note though the moment map \(\nu\) constructed before is an ``infinitesimal'' moment map, as it is the moment map of the group action on the tangent space of a fixed point.

In \cite{bradlowDimensionalReductionPerturbed2001}, Bradlow--Glazebrook--Kamber consider a holomorphic vector bundle \(\mcE \to M\), where \(M\) is a holomorphic fibre bundle over \(X\), with fibre \(F\). They consider the case when \(\mcE\) is given by an extension of two vector bundles. Letting \(\omega_k = \omega_X + k\omega_F\), they obtain the \emph{perturbed Hermite--Einstein equation}
\[\Lambda_{\omega_k}F + \mathfrak{d}_k(\beta) = 2\pi\lambda_k\id_E,\]
where \(\mathfrak{d}\) is the term representing the deformation \(\beta\). Despite visual similarities, there are many differences. In particular, the family Hermite--Einstein equation depends only on the horizontal component of the connection, and the vertical component of the deformation.

Finally, in \cite{bradlowRelativeHitchinKobayashiCorrespondences2003}, Bradlow--Garc\'ia-Prada--Mundet i Riera develop a general theory where one restricts to a subgroup of the gauge group. In particular, they prove a Hitchin--Kobayashi correspondence in such a setting. Without the deformation term, the family Hermite--Einstein equation may be viewed as a special case of the general equations considered in \cite{bradlowRelativeHitchinKobayashiCorrespondences2003} by considering the total space of the bundle \(\mcE \to X \times B\) and restricting to a subgroup of the gauge group. As far as we can tell, the deformation term cannot be incorporated into their framework. Geometrically, \cite{bradlowRelativeHitchinKobayashiCorrespondences2003} take a ``finite-dimensional'' perspective of considering \(\mcE\) as a bundle over \(X \times B\), whereas in this paper, we focus on the ``infinite-dimensional'' perspective of considering it as a family of bundles over \(X\) parametrised by \(B\).

\subsection{Outlook}

In general, when we have a moduli space or stack \(\mcM\), maps \(B \to \mcM\) correspond to families of semistable objects parametrised by \(B\). Thus, the ideas behind the work on optimal symplectic connections and from this paper should apply very generally: given a notion of a canonical metric, one should be able to define a corresponding canonical metric for families. 

Associated to canonical metrics is a moment map interpretation, as well as an algebro-geometric stability condition. For Hermite--Einstein metrics, the stability condition is given by \emph{slope stability}. We expect there to be an associated stability condition, and a Hitchin--Kobayashi type correspondence for family Hermite--Einstein metrics. As part of the moment map interpretation, we expect there to be a uniqueness statement for family Hermite--Einstein metrics, giving a precise way in which they produce canonical metrics.

\subsection{Outline}

We begin in \Cref{sec:prelim} with preliminary material on Hermite--Einstein metrics, Hermitian--Yang--Mills connections and the moment map interpretations. In \Cref{sec:def-he}, we study moment maps of linear actions of compact groups on vector spaces, and apply this to understanding deformations of Hermite--Einstein bundles. 
Next, in \Cref{sec:family-he-metrics}, we define the family Hermite--Einstein equation, both with and without deformations.
In \Cref{sec:space-of-vertically-he-metrics}, we study the geometry of the space of vertically Hermite--Einstein metrics, and make a formal analogy between the Hermite--Einstein and family Hermite--Einstein equations.

Next, in \Cref{sec:hym-total-space}, we show that given a solution of the family Hermite--Einstein equations, one can construct Hermitian--Yang--Mills connections on the total space, in adiabatic classes, proving \Cref{thm:intro-adiabatic-he}. Finally, in \Cref{sec:family-Hermite--Einstein-flow}, we show that the associated flow exists for all time, proving \Cref{thm:intro-hhe-flow}, and use this to solve the Dirichlet problem, proving \Cref{thm:intro-dirichlet-problem}.

\subsection{Ackowledgements}

I would like to thank my PhD supervisor Ruadha\'i Dervan for suggesting this problem, constant encouragement and many helpful discussions. I would also like to thank Leticia Brambila Paz, Siarhei Finski, Dylan Galt, Annamaria Ortu, Carlo Scarpa, Lars Martin Sektnan and Carl Tipler for helpful discussions. Finally, I would like to thank Joel Fine and Oscar Garc\'ia-Prada for discussions on vortex equations.

I was supported by a PhD studentship associated to Ruadha\'i Dervan's Royal Society University Research Fellowship (URF{\textbackslash}R1{\textbackslash}201041).

\section{Preliminaries}

\label{sec:prelim}

\subsection{Chern connections}

Let \(X\) be a complex manifold, \(\mcE = (E, \delbar)\) as holomorphic vector bundle on \(X\). For a Hermitian metric \(h\) on \(E\), let \(\grad\) denote the associated Chern connection.

We may decompose \(\grad = \grad^{1,0} + \grad^{0, 1}\) into its \((1, 0)\) and \((0, 1)\)-components. Supposing \(X\) has a K\"ahler metric \(\omega\), we may use the metrics \(\omega\) and \(h\) to define the adjoints to \(\grad^{1, 0}\) and \(\grad^{0, 1}\).

\begin{lemma}
    [{\cite[Section 3.2]{kobayashiDifferentialGeometryComplex2014}}] Let \(\grad\) be the Chern connection on \(\mcE\). Then
    \[(\grad^{1, 0})^* = i[\Lambda_\omega, \grad^{0, 1}] \quad\text{and}\quad (\grad^{0, 1})^* = -i[\Lambda_\omega, \grad^{1, 0}].\]
\end{lemma}

Using this, we define corresponding Laplacian operators for sections of \(E\) by
\[\Delta = \grad^*\grad, \quad\quad \Delta^{1, 0} = (\grad^{1, 0})^*\grad^{1, 0} \quad\text{and}\quad \Delta^{0, 1} = (\grad^{0, 1})^*\grad^{0, 1}.\]
These satisfy the relations
\begin{equation}
    \label{eqn:laplacian-relations}
        \Delta^{1, 0} + \Delta^{0, 1} = \Delta \quad\text{and}\quad \Delta^{1, 0} - \Delta^{0, 1} = i\Lambda_\omega F_h,
\end{equation}
where \(F_h = F_\grad\) is the curvature of the Chern connection.

As usual, we obtain induced metrics and induced connections on various associated bundles. On \(\End \mcE\), the metric is given by
\[h(\sigma, \tau) = \tr(\sigma\tau^*)\]
where \(\tau^*\) is the adjoint of \(\tau\) with respect to the metric \(h\). The associated Chern connection on \(\End \mcE\) behaves as follows with respect to taking adjoints.

\begin{lemma}
    \label{lem:grad-of-adjoint}
    Let \(\sigma \in \mcA^0(\End(E))\) be an endomorphism. Then
    \[\grad^{0, 1}(\sigma^*) = (\grad^{1, 0}\sigma)^*.\]
\end{lemma}

\subsection{Hermite--Einstein metrics}

Let \((X, \omega_X)\) be a compact K\"ahler manifold of dimension \(n\), and \(\mcE \to X\) a holomorphic vector bundle. Let \(h\) be a Hermitian metric on \(\mcE\), with Chern connection \(\grad\) and curvature \(F_h\).

\begin{definition}
    The Hermitian metric \(h\) is \emph{Hermite--Einstein} if
    \[i\Lambda_{\omega_X}F_h = c\id_E\]
    for some constant \(c \in \R\).
\end{definition}

\begin{remark}
    By Chern--Weil theory, \(c\) is a cohomological constant, given by
    \[c = \frac{2\pi n c_1(E) \cdot [\omega_X]^{n-1}}{\rank(E)[\omega_X]^n}.\]
\end{remark}

\begin{lemma}
    \label{lem:hermite-einstein-laplacians}
    Suppose \(h\) is a Hermite--Einstein metric on \(\mcE\). Let \(\Delta, \Delta^{1, 0}, \Delta^{0, 1}\) denote the total, \((1,0)\)- and \((0, 1)\)-Laplacians on \(\End \mcE\). Then
    \[\Delta = 2\Delta^{0, 1} = 2\Delta^{1, 0}.\]
    In particular, for any \(\sigma \in \mcA^0(\End \mcE)\), the following are equivalent:
    \begin{enumerate}[(i)]
        \item \(\grad\sigma = 0\),
        \item \(\grad^{1, 0}\sigma = 0\),
        \item \(\grad^{0, 1}\sigma = 0\).
    \end{enumerate}
\end{lemma}

\begin{proof}
    This follows from \Cref{eqn:laplacian-relations}, since the induced metric on \(\End \mcE\) is Hermite--Einstein, with constant \(c = 0\).
\end{proof}

The following result is well known, though we include the proof, as we will later need the decomposition of the Lie algebra. We refer to \cite{wangMomentMapFutaki2004} for a general moment map result.

\begin{proposition}
    \label{prop:he-reductive}
    Let \(\mcE\) be a holomorphic vector bundle, with a Hermite--Einstein metric \(h\) on \(\mcE\). Then \(G = \Aut(\mcE)\) is reductive, with maximal compact subgroup
    \[K = \Aut(\mcE) \cap U(E, h).\]
\end{proposition}

\begin{proof}
    Let \(\sigma \in H^0(X, \End \mcE)\) be a holomorphic endomorphism. By \Cref{lem:grad-of-adjoint,lem:hermite-einstein-laplacians}, \(\sigma^*\) is also holomorphic. Thus, we have a map
    \begin{align*}
        T \colon H^0(X, \End \mcE) &\to H^0(X, \End \mcE) \\
        \sigma &\mapsto \sigma^*
    \end{align*}
    such that \(T^2 = \id\). Thus, \(H^0(X, \End \mcE)\) decomposes into the \(+1\)- and \(-1\)-eigenspaces, which correspond to Hermitian and skew-Hermitian endomorphisms respectively. In other words, we have a decomposition
    \begin{equation}
        \label{eqn:decomposition-of-h0-end}
        \begin{split}
            H^0(X, \End \mcE) &= \left(H^0(X, \End \mcE) \cap \mcA^0(\End_{\mathrm{SH}}(E, h))\right) \\
            &\oplus \left(H^0(X, \End \mcE) \cap \mcA^0(\End_{\mathrm H}(E, h))\right).
        \end{split}
    \end{equation}
    But as \(\mfg = \Lie(G) = H^0(X, \End \mcE)\) and \(\mfk = \Lie(K) = H^0(X, \End \mcE) \cap \mcA^0(\End_{\mathrm{SH}}(E, h))\), we may rewrite \Cref{eqn:decomposition-of-h0-end} as
    \[\mfg = \mfk \oplus i\mfk\]
    as required.
\end{proof}

\subsection{Hermitian--Yang--Mills connections}

So far, we have taken the approach of fixing a holomorphic vector bundle, and varying the Hermitian metric on it. There is an alternative perspective to the problem, which involves fixing the Hermitian metric on a complex vector bundle, and varying the complex structure instead.

Let \((X, \omega_X)\) be a compact K\"ahler manifold, \(E \to X\) a smooth complex vector bundle with a Hermitian metric \(h\).

\begin{definition}
    A \(h\)-unitary connection \(\grad\) is \emph{Hermitian--Yang--Mills} if
    \[\begin{cases}
        i\Lambda_{\omega_X}F_\grad &= c\id_E \\
        F_\grad^{0, 2} &= 0.
    \end{cases}\]
\end{definition}

\begin{remark}
    The equation \(F_{\grad}^{0,2} = 0\) is equivalent to the condition that \((\grad^{0, 1})^2 = 0\), i.e. \(\grad^{0, 1}\) defines an integrable Dolbeault operator on \(E\).
\end{remark}

The gauge group action on the space of unitary connections is defined as follows. For \(g \in \GL(E)\) and a unitary connection \(\grad\) on \(E\), we define
\[g \cdot \grad = (g^{-1})^* \circ \grad^{1, 0} \circ g^* + g \circ \grad^{0, 1} \circ g^{-1}.\]

\begin{remark}
    This is the \emph{opposite} action to the one used in \cite{atiyahYangMillsEquationsRiemann1983,kobayashiDifferentialGeometryComplex2014,sektnanHermitianYangMillsConnections2024}, but agrees with \cite{donaldsonSelfdualYangMillsConnections1985}.
\end{remark}

The relationship with varying the Hermitian metric is as follows. Suppose \(h\) is a Hermitian metric on a holomorphic vector bundle \(\mcE = (E, \delbar)\), and \(\sigma\) is a positive-definite \(h\)-Hermitian endomorphism. Then
\begin{equation}
    \label{eqn:he-to-hym}
    F_{h\sigma, \delbar} = \sigma^{-1/2} \circ F_{h, \sigma^{1/2} \cdot \delbar} \circ \sigma^{1/2}.
\end{equation}

In particular, the solvability of the Hermite--Einstein and Hermitian--Yang--Mills equations are equivalent. As such, we will pass freely between the two situations, depending on which one is more convenient.

\subsection{Linearisations}

It will be useful to understand the linearisation of the Hermite--Einstein and Hermitian--Yang--Mills operators. The following is a simple computation.

\begin{proposition}
    \label{prop:he-hym-linearisation}
    Let \((E, \delbar, h)\) be a holomorphic vector bundle with a Hermitian metric.
    \begin{enumerate}[(i)]
        \item The linearisation of the map
        \begin{align*}
            \mcA^0(\End_{\mathrm H}(E, h)) &\to \mcA^0(\End_{\mathrm{H}}(E, h)) \\
            \sigma &\mapsto i\Lambda_{\omega_X}F_{h\exp(\sigma)} - c\id_E
        \end{align*}
        at \(\sigma = 0\) is
        \[\sigma \mapsto \Delta_h^{1, 0}\sigma.\]
        \item The linearisation of the map
        \begin{align*}
            \mcA^0(\End_{\mathrm H}(E, h)) &\to \mcA^0(\End_{\mathrm H}(E, h)) \\
            \sigma &\mapsto i\Lambda_{\omega_X}F_{\exp(\sigma) \cdot \delbar} - c\id_E
        \end{align*}
        at \(\sigma = 0\) is
        \[\sigma \mapsto \Delta\sigma.\]
    \end{enumerate}
\end{proposition}

\subsection{Hermitian--Yang--Mills as a moment map}

In this section, we recall the moment map picture for the Hermitian--Yang--Mills equations. Recall that if \((Y, \omega)\) is a symplectic manifold, \(K\) a Lie group acting on the left on \(\omega\) by symplectomorphisms, a map
\[\mu \colon Y \to \mfk^*\]
is called a \emph{moment map} if \(\mu\) is \(K\)-equivariant, and for any \(\xi \in \mfk\),
\[\dd\langle \mu, \xi \rangle = -\omega(X_\xi, \cdot),\]
where \(X_\xi\) is the infinitesimal action of \(\xi\) on \(Y\). Suppose \(\langle \cdot, \cdot \rangle_\mfk\) is a \(K\)-invariant inner product on \(\mfk\), which defines a map \(\mfk \to \mfk^*\). We also call a map \(\widetilde \mu \colon Y \to \mfk\) a \emph{moment map} if \(\langle \widetilde\mu, \cdot \rangle\) is a moment map as above.

Let \((X, \omega)\) be a compact K\"ahler manifold, \(E \to X\) a smooth complex vector bundle with a Hermitian metric \(h\). Recall that the space \(\mcA = \mcA(E, h)\) of unitary connections is an affine space modelled on \(\mcA^1(\End_{\mathrm{SH}}(E, h))\). In particular, for any \(\grad \in \mcA\), we may identify
\[T_\grad \mcA = \mcA^1(\End_{\mathrm SH}(E, h)).\]
Define a \(2\)-form \(\Omega\) on \(\mcA\) by
\[\Omega(\alpha, \beta) = -\int_X\Lambda_\omega\tr(\alpha \wedge \beta)\omega^n,\]
and an almost complex structure on \(\mcA\) by
\[J(\alpha) = -i\alpha^{1, 0} + i\alpha^{0, 1}.\]

\begin{proposition}
    \((\mcA, J, \Omega)\) is an infinite-dimensional K\"ahler manifold.
\end{proposition}

Let \(\mcK = U(E, h)\) denote the group of unitary automorphisms of \((E, h)\), which acts on \(\mcA\) by conjugation. The Lie algebra of \(\mcK\) is \(\Lie(\mcK) = \mcA^0(\End_{\mathrm{SH}}(E, h))\), and the adjoint action of \(\mcK\) on \(\Lie(\mcK)\) is by conjugation. We have a natural \(\mcK\)-invariant inner product on \(\Lie(\mcK)\), given by
\[\langle \xi, \eta \rangle_{\Lie(\mcK)} = -\int_X\tr(\xi \eta)\omega_X^n.\]

\begin{proposition}
    [\cite{atiyahYangMillsEquationsRiemann1983,donaldsonSelfdualYangMillsConnections1985}]
    The map
    \begin{align*}
        \mu \colon \mcA &\to \Lie(\mcK), \\
        \mu(\grad) &= -\Lambda_\omega F_\grad - ic\id_E,
    \end{align*}
    is a moment map for the \(\mcK\)-action on \(\mcA\).
\end{proposition}

\begin{remark}
    As the sign convention will play an important role in our results, we mention that there are many sign conventions which are used in the literature. In particular, the sign depends on
    \begin{enumerate}[(i)]
        \item whether the gauge group acts on connections on the left or on the right,
        \item whether K\"ahler metrics are given by \(g(u, v) = \omega(u, Jv)\) or \(\omega(Ju, v)\),
        \item whether moment maps satisfy \(\dd\langle\mu, \xi\rangle = \omega(X_\xi, \cdot)\) or \(\omega(\cdot, X_\xi)\).
    \end{enumerate}
    Each one of these choices will introduce a minus sign.
\end{remark}

Finally, we note that the linear isomorphism
\begin{align*}
    \mcA^1(\End_{\mathrm{SH}}(E, h)) &\to \mcA^{0, 1}(\End(E)), \\
    \alpha &\mapsto \alpha^{0, 1},
\end{align*}
gives an isomorphism of affine spaces between the space \(\mcA\) of unitary connections on \(E\), and the space \(\mcD\) of (not necessarily integrable) Dolbeault operators on \(E\). Moreover, this isomorphism is \(\mcK\)-equivariant, where \(\mcK\) acts on \(\mcD\) by
\[g \cdot \delbar = g \circ \delbar \circ g^{-1},\]
and so we may give \(\mcD\) the structure of a K\"ahler manifold with a \(\mcK\)-action, with moment map
\[\mu(\delbar) = -\Lambda_\omega F_{\delbar} - ic\id_E.\]

\subsection{Families of vector bundles}

Let \(X, B\) be complex manifolds. We are interested in holomorphic vector bundles over \(X\), parametrised by \(B\).

\begin{definition}
    A \emph{family of holomorphic vector bundles} on \(X\) parametrised by \(B\) is a holomorphic vector bundle on \(X \times B\).
\end{definition}

Given a family \(\mcE \to X \times B\), we will write \(\mcE_b = \mcE\vert_{X \times \{b\}} \to X\) for the corresponding bundle over \(X\).

Now suppose \((X, \omega_X), (B, \omega_B)\) are K\"ahler, of dimensions \(n\) and \(m\) respectively. On \(X \times B\), we may define two contraction operators \(\mcA^2(X \times B) \to \mcA^0(X \times B)\) as follows:
\begin{align*}
    \Lambda_V\alpha &= \frac{n\alpha \wedge \omega_X^{n-1} \wedge \omega_B^m}{\omega_X^n \wedge \omega_B^m} \\
    \Lambda_H\alpha &= \frac{m\alpha \wedge \omega_X^n \wedge \omega_B^{m-1}}{\omega_X^n \wedge \omega_B^m}.
\end{align*}

On \(X \times B\), we have a natural one-parameter family of K\"ahler metrics, given by
\[\omega_k = \omega_X + k\omega_B\]
for \(k > 0\).

\begin{lemma}
    The contraction operator with respect to \(\omega_k\) is given by
    \[\Lambda_k \coloneqq \Lambda_{\omega_k} = \Lambda_V + k^{-1}\Lambda_H.\]
\end{lemma}

Recall we defined the Laplacians
\[
    \Delta^{1, 0} = i\Lambda \grad^{0, 1}\grad^{1, 0}, \quad
    \Delta^{0, 1} = -i\Lambda \grad^{1, 0}\grad^{0, 1} \quad \text{and}\quad
    \Delta = \Delta^{1, 0} + \Delta^{0, 1}.
\]
Replacing \(\Lambda\) by \(\Lambda_k\), \(\Lambda_V, \Lambda_H\), we may define various Laplacian operators such as \(\Delta_V^{1, 0}, \Delta_V^{0, 1}\) etc., which satisfy analogues of \Cref{eqn:laplacian-relations}, as well as
\[\Delta_k = \Delta_V + k^{-1}\Delta_H\]
and the \((1, 0)\) and \((0, 1)\)-analogues.

In our adiabatic situation, the Einstein constant in the definition of the Hermite--Einstein and Hermitian--Yang--Mills equations will now be \(k\)-dependent.

\begin{lemma}
    The Einstein constant of \(\mcE\) with respect to \(\omega_k\) is given by
    \[c_k = c_V + k^{-1}c_H.\]
\end{lemma}

Similarly, the linearisations in \Cref{prop:he-hym-linearisation} will now be \(k\)-dependent. Moreover, we may also consider the linearisations with respect to \(\Lambda_V\) and \(\Lambda_H\).

\begin{proposition}
    Let \((E, \delbar, h)\) be a holomorphic vector bundle on \(X \times B\) with a Hermitian metric.
    \begin{enumerate}[(i)]
        \item The linearisation of the map
        \begin{align*}
            \mcA^0(\End_{\mathrm H}(E, h)) &\to \mcA^0(\End(E, h)) \\
            \sigma &\mapsto i\Lambda_V F_{h \exp(\sigma)} - c_V \id_E
        \end{align*}
        at \(\sigma = 0\) is given by
        \[\sigma \mapsto \Delta_V^{1, 0}\sigma.\]
        \item The linearisation of the map
        \begin{align*}
            \mcA^0(\End_{\mathrm H}(E, h)) &\to \mcA^0(\End_{\mathrm H}(E, h)) \\
            \sigma &\mapsto i\Lambda_V F_{\exp(\sigma)\cdot \delbar} - c_V\id_E
        \end{align*}
        at \(\sigma = 0\) is given by
        \[\sigma \mapsto \Delta_V\sigma.\]
    \end{enumerate}
    Moreover, the corresponding statements for \(\Lambda_H\) also hold.
\end{proposition}

\section{Deformation theory of Hermite--Einstein bundles}

\label{sec:def-he}

We next study the deformation theory of Hermite--Einstein bundles, following the general strategy as developed in \cite{szekelyhidiKahlerRicciFlowKpolystability2010,bronnleDeformationConstructionsExtremal2011} for constant scalar curvature K\"ahler metrics. We refer to \cite[Section 48]{krieglConvenientSettingGlobal1997} for more details on moment maps in infinite dimensions.

The general set-up is as follows. Let \(V\) be a Fr\'echet space. A \emph{symplectic form} on \(V\) is a continuous bilinear map
\[\Omega \colon V \times V \to \R\]
which is skew-symmetric and non-degenerate. Here, non-degenerate means that the induced map \(V \to V^*\) is injective.

Let \(K\) be a Lie group acting linearly on \((V, \Omega)\) by symplectomorphisms, and suppose \(\mu \colon V \to \mfk^*\) is a moment map for this action. Define a map
\begin{align*}
    \nu \colon V &\to \mfk^* \\
    \langle \nu(a), \xi \rangle &= -\frac{1}{2}\Omega(X_\xi(a), a),
\end{align*}
where
\[X_\xi(a) = \dv{t}\bigg\vert_{t=0}\exp(t\xi) \cdot a\]
is the infinitesimal action. Since the \(K\)-action on \(V\) is linear, so is the infinitesimal action \(X_\xi \colon V \to V\).

\begin{lemma}
    \(\nu\) is a moment map for the \(K\)-action on \(V\).
\end{lemma}

The result is standard, see for example \cite[Section 4.2]{chenCalabiFlowGeodesic2014}. Using this moment map, given a curve \(\alpha(s)\) in \(V\), we may expand \(\mu(\alpha(s))\) around \(s = 0\).

\begin{proposition}
    [{\cite{szekelyhidiKahlerRicciFlowKpolystability2010,inoueModuliSpaceFano2019}}]
    Let \(\alpha(s)\) be a smooth curve in \(V\), with \(\alpha(0) = 0\). Then
    \[\mu(\alpha(s)) = \mu(0) + s^2\nu(\alpha'(0)) + \order{s^3}.\]
\end{proposition}

\begin{proof}
    Fix \(\xi \in \mfk\). Let \(\mu^\xi = \langle \mu, \xi\rangle\) and \(F(s) = \mu^\xi(\alpha(s))\). Then
    \[\dv{F}{s} = \dd\mu^\xi_{\alpha(s)}(\alpha'(s)) = -\Omega(X_\xi(\alpha(s)), \alpha'(s)),\]
    which is zero at \(s = 0\). Next,
    \[\dv[2]{F}{s} = -\Omega(X_\xi(\alpha(s)), \alpha''(s)) - \Omega(X_\xi(\alpha'(s)), \alpha'(s))\]
    and so
    \[\dv[2]{F}{s}\bigg\vert_{s=0} = -\Omega(X_\xi(\alpha'(0)), \alpha'(0)) = 2\langle\nu(\alpha'(0)), \alpha'(0)\rangle.\]
\end{proof}

Let \((X, \omega)\) be a compact K\"ahler manifold, \(\mcE = (E, \delbar_0) \to X\) a holomorphic vector bundle, with a Hermite--Einstein metric \(h\). Let \(G = \Aut(\mcE)\), which is reductive and has maximal compact subgroup \(K = \Aut(\mcE) \cap U(E, h)\). Using the \(\mcK\)-equivariant inner product on \(\Lie(\mcK)\), we may define a projection \(\pr_\mfk \colon \Lie(\mcK) \to \mfk\), and
\[\delbar \mapsto -\pr_\mfk(\Lambda_\omega F_{\delbar}) - ic\id_E\]
is a moment map for the \(K\)-action on the space \(\mcD\) of Dolbeault operators. Thus, we obtain
\begin{corollary}
    \label{cor:deformation-expansion}
    Let \(\alpha(s)\) be a smooth curve in \(\mcD\) with \(\alpha(0) = \delbar_0\). Then
    \begin{equation}
        \label{eqn:deformation-expansion}
        \pr_\mfk(\Lambda_\omega F_{\alpha(s)}) = -ic\id_E - s^2\nu(\alpha'(0)) + \order{s^3}
    \end{equation}
    where
    \[\nu \colon \mcA^{0, 1}(\End(E)) \to \mfk\]
    is defined by
    \[\langle\nu(a), \xi\rangle = -\frac{1}{2}(X_\xi(a), a) = \frac{1}{2}\Omega([\xi, a], a).\]
\end{corollary}

We will later make use of the following result.

\begin{lemma}
    For all \(a \in \mcA^{0, 1}(\End(E))\),
    \[\int_X\tr(\nu(a))\omega^n = 0.\]
\end{lemma}

\begin{proof}
    This follows from a direct computation:
    \[\int_X\tr(\nu(a))\omega_X^n = -i\langle \nu(a), i\id_E\rangle = \frac{i}{2}\Omega([i\id_E, a], a) = 0.\]
\end{proof}

\subsection{Kuranishi theory}

\label{subsec:kuranishi-theory}

The results of this paper will not directly make use of Kuranishi theory, though it will be useful to have the results in mind. We refer to \cite{kobayashiDifferentialGeometryComplex2014} for more details on Kuranishi theory for holomorphic vector bundles, and \cite[Section 4]{ortuOptimalSymplecticConnections2023} for Kuranishi theory in families.

Let \((X, \omega_X)\) be a compact K\"ahler manifold, \(\mcE = (E, \delbar) \to X\) a holomorphic vector bundle, with Hermitian metric \(h\). Suppose \((\mcE, h)\) is Hermite--Einstein. Deformations of \(\mcE\) are parametrised by an elliptic complex
\[\begin{tikzcd}
    \mcA^{0}(\End E) \arrow{r}{\delbar} & \mcA^{0, 1}(\End E) \arrow{r}{\delbar} & \mcA^{0, 2}(\End E).
\end{tikzcd}\]
Let \(H^{0, 1}(X, \End \mcE)\) denote the space of \(\delbar\)-harmonic \((0, 1)\)-forms with value in \(\End E\). 

\begin{theorem}
    [{\cite{kuranishiNewProofExistence1965}}] There exists an open neighbourhood \(U\) of \(0 \in H^{0, 1}(X, \End \mcE)\), and a holomorphic map
    \[\Psi \colon U \to \mcA^{0, 1}(\End E)\]
    such that \(\Psi(0) = 0\), \(\dd\Psi_0 = \id\), and for any \(a \in \mcA^{0, 1}(\End E)\) sufficiently close to \(0\), with \(\delbar + a\) integrable, there exists a \(u \in U\) and a gauge transformation \(f \in \GL(E)\) such that \(f \cdot (\delbar + a) = \delbar + u\).
\end{theorem}

\begin{remark}
    In fact, the Kuranishi map \(\Psi\) is \(K = \Aut(\mcE) \cap U(E, h)\)-equivariant. We refer to \cite{szekelyhidiKahlerRicciFlowKpolystability2010,chenCalabiFlowGeodesic2014,buchdahlPolystableBundlesRepresentations2022} for more details on this.
\end{remark}

Following \cite{ortuOptimalSymplecticConnections2023}, we will now develop Kuranishi theory for families. Suppose \((X, \omega_X)\) and \((B, \omega_B)\) are compact K\"ahler manifolds. Let \(\mcE \to X \times B\) be a family of holomorphic vector bundles, \(h\) a Hermitian metric on \(\mcE\), such that each \((\mcE_b, h_b)\) is Hermite--Einstein.

The integrability condition \(\delbar \circ \delbar = 0\) for \(\mcE\) may be written as
\[\delbar_V \circ \delbar_V = 0, \quad \delbar_H \circ \delbar_H = 0 \quad\text{and}\quad \delbar_V \circ \delbar_H + \delbar_H \circ \delbar_V = 0.\]
We will consider deformations of \(\mcE\) which preserve \(\delbar_H\). Define the space
\[H_V = \{\alpha \in \mcA^{0}(T^*X^{0, 1} \otimes \End E) \mid \delbar_V\alpha = 0, \delbar_V^*\alpha = 0, \delbar_H\alpha = 0\}.\]
This is finite-dimensional, as the operator \(\delbar_V^*\delbar_V + \delbar_V\delbar_V^* + \delbar_H^*\delbar_H\) is elliptic on \(T^*X^{0, 1} \otimes \End E\).

\begin{theorem}
    There exists an open neighbourhood \(U\) of \(0 \in H_V\), and a holomorphic map
    \[\Psi_V \colon U \to \mcA^{0, 1}(\End E)\]
    such that \(\Psi_V(0) = 0\), \((\dd\Psi_V)_0 = \id\), and for any \(\alpha \in \mcA^{0, 1}(\End E)\) sufficiently close to \(0\), with \(\delbar + \alpha\) integrable and \(\alpha_H = 0\), there exists \(u \in U\), such that for all \(b \in B\), there exists a gauge transformation \(f_b \in \GL(E_b)\) such that \(f_b \cdot (\delbar_b + \alpha_b) = \delbar_b + \Psi_V(u)\vert_b\).
\end{theorem}

The proof is similar to the non-family case, and we refer to \cite[Theorem 4.6]{ortuOptimalSymplecticConnections2023} for more details. We note that as \((\dd\Psi_V)_0 = \id\), first order deformations arising from the Kuranishi map satisfy the gauge fixing condition \(\delbar_V^*\alpha = 0\).

\section{Family Hermite--Einstein metrics}

\label{sec:family-he-metrics}

With everything set-up, we can now define the family Hermite--Einstein equation. Let \((X, \omega_X)\) be compact K\"ahler, and let \((B, \omega_B)\) be a K\"ahler manifold.  We will primarily be interested in the case \(B\) is compact, but we do not assume this for the moment. Suppose \(\mcE^0 = (E, \delbar_0) \to X \times B\) is a holomorphic vector bundle, with a Hermitian metric \(h\) which is vertically Hermite--Einstein, in the sense that each \((\mcE_b^0, h_b)\) is Hermite--Einstein.

For \(b \in B\), we let \(\delbar_b\) denote the Dolbeault operator on \(\mcE_b\). By elliptic operator theory, the kernel of \(\delbar_b\) is finite-dimensional, and so we may define an \(L^2\)-projection map
\[\pi_{h, b} \colon \mcA^0(\End(E_b)) \to H^0(\End(\mcE_b)),\]
for a fixed Hermitian metric \(h\) on \(\mcE\). Combining these fibrewise projection maps, we obtain a map
\[\pi_h \colon \mcA^0(\End(E)) \to \mcA^0(\End(E)).\]
Similarly, we let \(p_{h, b}(\sigma)\) denote the trace-free part of \(\pi_{h, b}(\sigma)\), so that this defines a map \(p_h \colon \mcA^0(\End(E)) \to \mcA^0(\End(E))\).

\begin{remark}
    In \Cref{sec:space-of-vertically-he-metrics}, we will build a holomorphic vector bundle \(\mcF\) on \(B\), and show that \(\pi_h\) defines a map \(\mcA^0(\End(E)) \to \mcA^0(\mcF)\), under the assumption that \(\dim(H^0(\End(\mcE_b)))\) is constant.
\end{remark}

Given a curve \(\alpha\) in \(\mcA^{0, 1}(\End(E))\) with \(\delbar_0 + \alpha(s)\) integrable for all \(s\), we may decompose \(\alpha = \alpha_V + \alpha_H\). The corresponding Dolbeault operator is then
\[\delbar_0 + \alpha = (\delbar_{0, V} + \alpha_V) + (\delbar_{0, H} + \alpha_H).\]
Let \(\alpha_b = \alpha_V\vert_{X \times \{b\}}\). Then for each \(b \in B\), by \Cref{cor:deformation-expansion}, we have an expansion
\[\pr_{h, \mfk_b}\left(\Lambda_{\omega_X}F_{\delbar_0 + \alpha_b(s)}\right) = -ic\id_E - s^2\nu_b(\alpha_b'(0)) + \order{s^3}.\]

We note that the left hand side depends smoothly on \(b \in B\), and so
\[b \mapsto s^2\nu_{h, b}(\alpha_b'(0))\]
is also smooth. We will write \(\nu_h(\alpha_V'(0))\) for this endomorphism. Working globally, this means that
\[p_h(i\Lambda_V F_{\delbar_0 + \alpha(s)}) = -s^2i\nu_h(\alpha_V'(0)) + \order{s^3}.\]

Fix a first order vertical deformation \(a \in \mcA^0(T^*X^{0, 1} \otimes \End(E))\).

\begin{definition}
    A Hermitian metric \(h\) on \(E\) which is vertically Hermite--Einstein with respect to \(\delbar_0\) satisfies the \emph{family Hermite--Einstein equation (with respect to the first order deformation \(a\))} if
    \[p_h(i\Lambda_H F_h) - i\lambda \nu_h(a) = 0\]
    for some constant \(\lambda > 0\).
\end{definition}

We view this as the correct notion of a canonical metric on a family of semistable vector bundles (where \(a\) is the associated first order deformation from the corresponding polystable family), much as the Hermite--Einstein equation defines a notion of a canonical metric on a holomorphic vector bundle. We note, however, that the equation depends only on the first order deformation \(a\).

\begin{example}
    When \(X\) is a point, \(\pi_h = \id\), and so the family Hermite--Einstein equation is simply the \emph{weak} Hermite--Einstein equation. If \(B\) is compact, then the solvability of this equation is equivalent to the solvability of the Hermite--Einstein equation \cite[Proposition 4.1.4]{kobayashiDifferentialGeometryComplex2014}.
\end{example}

\begin{remark}
    The constant \(\lambda\) is a coupling constant. In the adiabatic setting, \(\lambda\) controls the rate at which the metric \(\omega_k = \omega_X + k\omega_B\) grows, versus the rate of the degeneration to a family of Hermite--Einstein bundles.
\end{remark}

\begin{remark}
    We will construct in \Cref{sec:space-of-vertically-he-metrics} a holomorphic vector bundle \(\mcF\) on \(B\), under the assumption that \(\dim(\End \mcE_b)\) is constant, such that the family Hermite--Einstein equation is a partial differential equation for sections of \(\mcF\).
\end{remark}

\section{The space of vertically Hermite--Einstein metrics}

\label{sec:space-of-vertically-he-metrics}

In this section, we will develop the Riemannian geometry of the space of fibrewise Hermite--Einstein metrics. The corresponding results for families of cscK manifolds were developed by Hallam \cite{hallamGeodesicsSpaceRelatively2023}. We refer to \cite[Sections 6.1-6.2]{kobayashiDifferentialGeometryComplex2014} for more details about the space of Hermitian metrics.

\subsection{The space of Hermitian metrics}

Let \((B, \omega_B)\) be a compact K\"ahler manifold, and let \(E\) be a smooth complex vector bundle on \(B\). 
Let \(\msH\) denote the space of Hermitian forms on \(E\), and \(\msH^+ \subseteq \msH\) the space of positive definite Hermitian forms, which is an open convex cone in \(\msH\). For \(h \in \msH^+\), we can then identify
\[T_h\msH^+ = \msH.\]
On this cone, we have a Riemannian metric
\[\langle u, v \rangle_h = \int_B\tr(h^{-1}u \circ h^{-1}v)\omega_B^m.\]
We note that this Riemannian metric is \(\GL(E)\)-invariant, and that it is the metric which we obtain on \(\msH^+\) by viewing it as a homogeneous space. In particular, this is non-positively curved, with geodesics given by
\[h(t) = h_0 \exp(ta)\]
for \(a \in \mcA^0(\End_{\mathrm H}(E, h_0))\), and the map
\begin{equation}
    \label{eqn:Hermitian-metrics-diffeomorphism}
    \begin{split}
        \mcA^0(\End_{\mathrm H}(E, h_0)) &\to \msH^+, \\
            a &\mapsto h_0 \exp(a),
    \end{split}
\end{equation}
is a diffeomorphism.

\subsection{The space of Hermite--Einstein metrics}

Now suppose \(\mcE = (E, \delbar)\) is a holomorphic vector bundle on a compact K\"ahler manifold \((X, \omega_X)\). It then makes sense to consider the space \(\msH^+_{\mathrm{HE}} \subseteq \msH^+\) of Hermite--Einstein metrics.

\begin{proposition}
    Let \(h\) be a Hermitian metric on \(E\). Suppose \(h\) is Hermite--Einstein. Let \(\sigma\) be a \(h\)-Hermitian endomorphism. Then \(h\exp(\sigma)\) is Hermite--Einstein if and only if \(\sigma\) is holomorphic.
\end{proposition}

\begin{proof}
    This follows using the convexity of the Donaldson functional, see \cite[Section 6.3]{kobayashiDifferentialGeometryComplex2014} for more details.
\end{proof}

The following restatement will be the one which we will later use.

\begin{corollary}
    Let \(h\) be a Hermite--Einstein metric on \(\mcE\). Let \(\sigma\) be a positive-definite \(h\)-Hermitian endomorphism. Then \(h\sigma\) is Hermite--Einstein if and only if \(\sigma\) is holomorphic.
\end{corollary}

From this, we obtain that \(\msH^+_{\mathrm{HE}}\) is a finite-dimensional totally geodesic submanifold of \(\msH^+\). In particular, the geodesics are of the form
\[h(t) = h_0\exp(ta)\]
with
\[a \in \mcA^0(\End_H(E, h_0)) \cap H^0(X, \End \mcE).\]
The same logic implies that we have a diffeomorphism
\[\msH^+_{\mathrm{HE}} \cong \mcA^0(\End_{\mathrm H}(E, h_0)) \cap H^0(X, \End \mcE)\]
for any Hermite--Einstein metric \(h_0\).

\subsection{The space of vertically Hermite--Einstein metrics}

\label{subsec:space-of-vertically-he-metrics}

We will now consider the case of families. Let \((X, \omega_X)\) and \((B, \omega_B)\) be compact K\"ahler manifolds, \(\mcE \to X \times B\) a holomorphic vector bundle. We may then consider the space \(\msH^+_{\mathrm V}\) of Hermitian metrics on \(\mcE\) which are Hermite--Einstein on each \(\mcE_b\). We would like to use the same considerations as above, replacing \(H^0(\End \mcE)\) with \(\ker(\delbar_V)\). A priori, \(b \mapsto \ker(\delbar_b)\) does not define a vector bundle on \(B\). We thus make the following assumption.

\begin{assumption}
    \label{assumption:constant-dimension}
    The map
    \[b \mapsto \dim H^0(X, \End(\mcE_b))\]
    is constant.
\end{assumption}

Under \Cref{assumption:constant-dimension}, we have a smooth vector bundle \(\mcF \to B\), whose fibre at \(b \in B\) is
\[\mcF_b = H^0(X, \End(\mcE_b)).\]
The vector bundle \(\mcF\) is in fact holomorphic, as it is given by the direct image
\[\mcF = (\pr_B)_*\left(\End \mcE\right),\]
where the Dolbeault operator on \(\mcF\) is given by \(\delbar_H\). 

With all of this in mind, \(\msH^+_{\mathrm V}\) is a totally geodesic submanifold of \(\msH^+\), and we have a diffeomorphism
\[\msH^+_{\mathrm V} \cong \mcA^0(\End_H(E, h_0)) \cap \ker(\delbar_V),\]
for any vertically Hermite--Einstein metric \(h_0\) on \(\mcE\). Our next goal is to express the right hand side as the global sections of a real subbundle of \(\mcF\).

From the proof of \Cref{prop:he-reductive}, we have a decomposition
\[\mcF_b = H_b \oplus S_b,\]
corresponding to Hermitian and skew-Hermitian endomorphisms respectively. Since the Hermitian metric, and so the conjugation map, varies smoothly with \(b \in B\), this decomposition gives us a decomposition
\[\mcF = H \oplus S\]
of smooth vector bundles over \(B\). Using this, we obtain a diffeomorphism
\begin{equation}
    \label{eqn:Hermite--Einstein-vertical-metrics-diffeomorphism}
    \msH^+_{\mathrm V} \cong \mcA^0(H).
\end{equation}

Finally, we note that we may also express \(\msH^+_V\) as a homogeneous space, where
\begin{align*}
    G &= \GL(E) \cap \ker(\delbar_V), \\
    K &= U(E, h_0) \cap \ker(\delbar_V),
\end{align*}
which have Lie algebras \(\mcA^0(\mcF)\) and \(\mcA^0(S)\) respectively.

\subsection{Comparisons between Hermitian metrics and vertically Hermite--Einstein metrics}

The pictures in \Cref{eqn:Hermitian-metrics-diffeomorphism} and \Cref{eqn:Hermite--Einstein-vertical-metrics-diffeomorphism} look quite similar. In both cases, the space of metrics is diffeomorphic to the space of sections of a bundle. We will now make this analogy more precise. This will be used in \Cref{sec:family-Hermite--Einstein-flow} where we show that the family Hermite--Einstein flow exists for all time.

Recall we defined a projection map
\[\pi_{h, b} \colon \mcA^0(\End(E_b)) \to H^0(\End(\mcE_b)),\]
for a fixed Hermitian metric \(h\) on \(\mcE\). Combining these fibrewise projection maps, we obtain a map
\[\pi_h \colon \mcA^0(\End(E)) \to \mcA^0(\mcF).\]

\begin{lemma}
    The projection map \(\pi_h \colon \mcA^0(\End(E)) \to \mcA^0(\mcF)\) restricts to maps
    \begin{align*}
        \mcA^0(\End_{\mathrm H}(E, h)) &\to \mcA^0(H) \\
        \mcA^0(\End_{\mathrm{SH}}(E, h)) &\to \mcA^0(S).
    \end{align*}
\end{lemma}

\begin{proof}
    Since the projection map is defined fibrewise, and we know that \(\mcF_b = H_b \oplus S_b\), the result follows from the fact that Hermitian and skew-Hermitian endomorphisms are orthogonal with respect to the induced \emph{real} \(L^2\)-inner product on \(\End(E)\).
\end{proof}

Given any Hermitian metric \(h\) on \(\mcE\), we have an induced \(L^2\)-Hermitian metric \(h_\mcF\) on \(\mcF\), defined by
\[h_\mcF(\sigma, \tau) = \int_{X \times B/B}h(\sigma, \tau)\omega_X^n.\]
Define an operator
\[\grad_\mcF \colon \mcA^0(\mcF) \to \mcA^1(\mcF)\]
by
\[(\grad_\mcF \sigma)(v) = \pi_h((\grad_H \sigma)(v))\]
for any \(\sigma \in \mcA^0(\mcF)\) and vector field \(v\) on \(B\).

\begin{lemma}
    \(\grad_\mcF\) is the Chern connection on \(\mcF\) with respect to the Hermitian metric \(h_\mcF\).
\end{lemma}

\begin{proof}
    We need to check that it defines a connection, that it is compatible with the complex structure, and that it is compatible with the Hermitian metric \(h_\mcF\). It is easy to see that
    \[(\grad_\mcF \sigma)(v) = \pi_h((\grad_H \sigma)(v))\]
    satisfies \(C^\infty(B)\)-linearity in \(v\) and the Leibniz rule. Next, since \(\grad_H\) is the horizontal component of the Chern connection on \(\End \mcE\), the \((0, 1)\)-part is the operator \(\delbar_H\). Now we note that by integrability,
    \[\delbar_V \circ \delbar_V + \delbar_V \circ \delbar_H + \delbar_H \circ \delbar_V + \delbar_H \circ \delbar_H = 0.\]
    Splitting by type,
    \[\delbar_V \circ \delbar_V = 0, \quad \delbar_V \circ \delbar_H = -\delbar_H \circ \delbar_V, \quad\text{and}\quad \delbar_H \circ \delbar_H = 0.\]
    The second equation then implies that \(\delbar_H\) sends elements of \(\ker(\delbar_V)\) to itself. Recalling that the Dolbeault operator on \(\mcF\) is given by \(\delbar_H\), this gives the compatibility with the complex structure.

    Finally, we need to check that \(\grad_\mcF\) is compatible with the Hermitian metric \(h_\mcF\). Let \(\sigma_1, \sigma_2 \in \mcA^0(\mcF)\), and let \(v\) be a vector field on \(B\). Then
    \begin{align*}
        \dd h_\mcF(\sigma_1, \sigma_2) &= \dd\int_{X \times B/B}h(\sigma_1, \sigma_2)\omega_X^n, \\
        &= \int_{X \times B/B}\dd_H h(\sigma_1, \sigma_2)\omega_X^n, \\
        &=\int_{X \times B/B}h(\grad_H \sigma_1, \sigma_2)\omega_X^n + \int_{X \times B/B}h(\sigma_1, \grad_H \sigma_2)\omega_X^n.
    \end{align*}
    Contracting with the vector field \(v\), we obtain that
    \begin{align*}
        \dd h_\mcF(\sigma_1, \sigma_2)(v) &= \int_{X \times B/B}\left(h((\grad_H\sigma_1)v, \sigma_2) + h(\sigma_1, (\grad_H\sigma_2)(v))\right)\omega_X^n, \\
        &= \int_{X \times B/B}\left(h((\grad_F\sigma_1)(v), \sigma_2) + h(\sigma_1, (\grad_\mcF\sigma_2)(v))\right)\omega_X^n, \\
        &= h_\mcF((\grad_\mcF\sigma_1)(v), \sigma_2) + h_\mcF(\sigma_1, (\grad_\mcF\sigma_2)(v))
    \end{align*}
    as required.
\end{proof}

With all of this in mind, we can now state the formal analogy between the space of Hermitian metrics and the space of vertically Hermite--Einstein metrics.

\begin{center}
    \captionof{table}{Analogy between Hermitian and vertically Hermite--Einstein metrics\label{tab:analogy}}
    
    \begin{tabular}{c c c}
        & Hermitian & vertically Hermite--Einstein \\
        \hline
        Geometry & \(\mcE \to B\) & \(\mcE \to X \times B\) \\
        Endomorphism bundle & \(\End \mcE \to B\) & \(\mcF \to B\) \\
        Fibrewise inner products & \(\End_{\mathrm H}(\mcE_b, h)\) & \(H_b\) \\
        Connection on endomorphism bundle & \(\grad_{\End \mcE}\) & \(\grad_\mcF\) \\
        Inner product on endomorphisms & \(h_{\End}\) & \(h_\mcF\).
    \end{tabular}
\end{center}

\begin{remark}
    With this picture in mind, we may think of \Cref{assumption:constant-dimension} as a locally free type assumption. We expect that this assumption may be removed in the theory of family Hermite--Einstein metrics, analogous to the generalisation of the Hermite--Einstein theory to reflexive sheaves by Bando--Siu \cite{bandoStableSheavesEinsteinHermitian1994}.
\end{remark}

\subsection{Riemannian geometry of the space of vertically Hermite--Einstein metrics}

To conclude this section, we will consider the natural Riemannian metric on \(\msH^+_{\mathrm{V}}\), given by the fibrewise homogeneous space structure. Let \(\mcE \to X \times B\) be a holomorphic vector bundle, and let \(h_0\) be a vertically Hermite--Einstein metric on \(\mcE\). Define
\begin{align*}
    G_b &= \Aut(\mcE_b), \\
    K_b &= G_b \cap \mathrm U(E_b, h_0).
\end{align*}
Let \(\mfg_b\) and \(\mfk_b\) denote their respective Lie algebras. On \(\mfk_b\), we have a \(K_b\)-invariant inner product, given by
\[\langle u_1, u_2 \rangle_{\mfk_b} = -\int_X\tr(u_1u_2)\omega_X^n.\]
Letting \(\pi_b \colon G_b \to G_b/K_b\) denote the quotient map, we have an induced Riemannian metric on \(G_b/K_b\), given by
\[\langle v_1, v_2 \rangle_{\pi_b(g)} = \langle u_1, u_2\rangle_{\mfk_b},\]
where \(v_i = \dd\pi_g(gi u_i)\).

\begin{lemma}
    \label{lem:metric-on-homogeneous-spaces}
    Let \(d\) denote the metric on \(\msH^+_{\mathrm{V}}\) induced by the embedding into \(\msH^+\), and let \(d_b\) denote the metric on the space of Hermite--Einstein metrics on \(\mcE_b\). Then
    \[d(\sigma_1, \sigma_2)^2 = \int_B d_b(\sigma_1, \sigma_2)^2\omega_B^m.\]
\end{lemma}

\begin{proof}
    This follows from the fact that a geodesic in \(\msH^+_{\mathrm V}\) gives a geodesic in the space of Hermitian metrics on \(\mcE\), which in turn gives geodesics in the space of Hermitian metrics on \(\mcE_b\) for all \(b \in B\), which is a geodesic in the space of Hermite--Einstein metrics on \(\mcE_b\). The distance is then given by the speed of the geodesic \(\gamma\) with \(\gamma(0) = \sigma_1\) and \(\gamma(1) = \sigma_2\). The result then follows by comparing the inner products on \(\mfk_b\) and on \(\mcA^0(S)\).
\end{proof}

\subsection{Decomposition of the space of Hermitian endomorphisms}

\label{subsec:decomposition-hermitian-endomorphisms}

Using all that we have shown, we obtain a decomposition of the space of Hermitian endomorphisms on \((E, h)\) into
\[\mcA^0(\End_{\mathrm H}(E, h)) = \mcA^0(H) \oplus R\]
where \(R_b\) is the \(L^2\)-orthogonal complement of \(H_b\). Moreover, \(\id_E\) defines a section of \(H\), and so we may further decompose \(\mcA^0(H)\), to obtain
\begin{equation}
    \label{eqn:decomposition-of-Hermitian-endomorphisms}
    \mcA^0(\End_{\mathrm H}(E, h)) = C^\infty(B, \R)\id_E \oplus \mcA^0(H_0) \oplus R,
\end{equation}
where we have an orthogonal splitting
\[H = \langle \id_E \rangle \oplus H_0.\]
We denote by \(p_h \colon \mcA^0(\End_\mathrm{H}(E, h)) \to \mcA^0(H_0)\) the resulting orthogonal projection map.

\section{Hermitian--Yang--Mills connections on the total space}

\label{sec:hym-total-space}

In this section, we will construct Hermitian--Yang--Mills connections on the total space of a family of holomorphic vector bundles in adiabatic classes, under the assumption of the solvability of the family Hermite--Einstein equation.

Let \((X, \omega_X)\) and \((B, \omega_B)\) be compact K\"ahler manifolds. Suppose \(\mcE^0 = (E, \delbar_0) \to X \times B\) is a holomorphic vector bundle, with a Hermitian metric \(h\) which is vertically Hermite--Einstein. Let \(\alpha\) be a curve in \(\mcA^{0,1}(\End(E))\), and we write \(\delbar_s = \delbar_0 + \alpha(s)\). We will assume that \(\delbar_s\) is integrable for all \(s\). Let
\[a = \pdv{\alpha}{s}\bigg\vert_{s=0}\]
denote the associated first order deformation.

Recall \Cref{assumption:constant-dimension}, which states that
\[b \mapsto \dim(H^0(X, \End(\mcE_b^0)))\]
is constant. Let \(G = \Aut(\mcE^0)\) denote the automorphism group of the holomorphic vector bundle \(\mcE^0\), which has Lie algebra \(\mfg = H^0(\End(\mcE^0))\), and which acts on \(\mcA^0(T^*X \otimes \End(E))\) by conjugation. We define
\[G_a = \{g \in G \mid ga_Vg^{-1} = a_V\}\]
to be the stabiliser of \(a_V\), which has Lie algebra
\[\mfg_a = \{\xi \in \mfg \mid [\xi, a_V] = 0\}.\]
We note that \(\dim \mfg_a \ge 1\), since \(\id_E\) defines an element of \(\mfg_a\).

\begin{theorem}
    \label{thm:hym-total-space}
    Under the above hypotheses, suppose further that 
    \begin{enumerate}[(i)]
        \item \(h\) satisfies the family Hermite--Einstein equation with respect to the first order deformation \(a\), 
        \item \((E, \delbar_s)\) is simple for \(\abs{s} \ll 1\),
        \item \(\dim(\mfg_a) = 1\),
        \item \(\delbar_{V, 0}^*a_V = 0\).
    \end{enumerate}
    Then for all \(k \gg 0\), there exists Hermite--Einstein metrics \(h_k\) on \((E, \delbar_s)\) with respect to \(\omega_k\), where \(s^2 = \lambda k^{-1}\).
\end{theorem}

It will be more convenient to construct Hermitian--Yang--Mills connections, and so we will take the perspective of Hermitian--Yang--Mills connections throughout.

\begin{remark}
    Geometrically, we view the family Hermite--Einstein equation as an equation for a family of polystable bundles, along with a first order deformation. As such, \(\mfg_a\) is the natural Lie algebra controlling automorphisms of this, and we will show that \(\mfg_a\) is the kernel of the linearisation of the family Hermite--Einstein equation. In particular, we do not assume that \((E, \delbar_s)\) and \((E, \delbar_s')\) are isomorphic for \(s \ne s'\), \(0 < \abs{s}, \abs{s'} \ll 1\).
\end{remark}

\begin{remark}
    Using Kuranishi theory as sketched in \Cref{subsec:kuranishi-theory}, and choosing a basis of \(\mfg\), the condition that \(\dim(\mfg_a) = 1\) may be phrased as a matrix depending on \(a\) having maximal rank. In particular, this is an open condition on the first order deformation \(a\). Similarly, the condition \(\delbar_{V, 0}^*a_V = 0\) is a gauge-fixing condition, which again is satisfied for a generic choice of first order deformation \(a\) in the Kuranishi space. The key assumption in \Cref{thm:hym-total-space} is the existence of a family Hermite--Einstein metric.
\end{remark}

The general strategy follows the work of \cite{hongConstantHermitianScalar1999,fineConstantScalarCurvature2004a,dervanOptimalSymplecticConnections2021,ortuOptimalSymplecticConnections2023,sektnanHermitianYangMillsConnections2024}. In \Cref{subsec:approximate-solutions}, we construct approximate solutions to all orders, by understanding the linearisation of the family Hermite--Einstein equation. Following this, in \Cref{subsec:perturb-to-genuine-solution} we use the quantitative implicit function theorem to perturb our approximate solutions to genuine solutions.

\subsection{Expansion of the Hermitian--Yang--Mills operator}

\begin{proposition}
    Suppose we set \(s^2 = \lambda k^{-1}\). Then
    \[p_h(i\Lambda_k F_{\delbar_s}) = p_h(i\Lambda_V F_{\delbar_0}) + k^{-1}\left(p_h(i\Lambda_H F_{\delbar_0}) - i\lambda \nu_h(a)\right) + \order{k^{-3/2}}.\]
\end{proposition}

\begin{proof}
    Recall that \(\Lambda_k = \Lambda_V + k^{-1}\Lambda_H\), and so 
    \[p_h(i\Lambda_k F_{\delbar_s}) = p_h(i\Lambda_V F_{\delbar_s}) + k^{-1}p_h(i\Lambda_H F_{\delbar_s}).\]
    By definition of \(\nu_h\), we obtain that
    \[p_h(i\Lambda_V F_{\delbar_s}) = p_h(i\Lambda_V F_{\delbar_0}) - k^{-1}i\lambda \nu_h(a) + \order{k^{-3/2}}.\]
    As we have an \(\order{k^{-1/2}}\) perturbation,
    \[p_h(i\Lambda_H F_{\delbar_s}) = p_h(i\Lambda_H F_{\delbar_0}) + \order{k^{-1/2}}.\]
    Combining these two expansions gives the result.
\end{proof}

\subsection{Approximate solutions}

\label{subsec:approximate-solutions}

In this section, we will construct approximate solutions to the Hermite--Einstein equation on the total space. The notation follows \Cref{subsec:decomposition-hermitian-endomorphisms}.

\begin{proposition}
    \label{prop:approximate-solutions}
    Suppose \Cref{assumption:constant-dimension} holds, \(\dim(\mfg_a) = 1\). Then for each nonnegative integer \(r\), there exist
    \begin{enumerate}[(i)]
        \item \(\phi_1, \dots, \phi_r \in C^\infty(B)\),
        \item \(\sigma_1, \dots, \sigma_r \in \mcA^0(H_0)\),
        \item \(\tau_1, \dots, \tau_r \in R\),
        \item and constants \(\gamma_0, \dots, \gamma_r \in \R\),
    \end{enumerate}
    such that if we let
    \[f_{k, 0} = \id\]
    and
    \[f_{k, r} = \exp(k^{-r/2}\phi_r)\exp(k^{-r/2}\tau_r)\exp(k^{-r/2 + 1}\sigma_r)f_{k, r-1},\]
    then
    \[i\Lambda_k F_{f_{k, r} \cdot \delbar_s} = \left(\sum_{j=0}^r k^{-j/2}\gamma_j\right)\id_E + \order{k^{-(r + 1)/2}}.\]
\end{proposition}

The proof of \Cref{prop:approximate-solutions} will be by induction on \(r\). More precisely, we have an expansion of \(i\Lambda_k F_{k, r \cdot \delbar_s}\) of the form
\[i\Lambda_k F_{f_{k, r} \cdot \delbar_s} = \left(\sum_{j=0}^r k^{-j/2}\gamma_j\right)\id_E + k^{-(r + 1)/2}\psi_{r + 1} + \order{k^{-(r + 2)/2}}.\]
From our decomposition in \Cref{eqn:decomposition-of-Hermitian-endomorphisms}, we obtain that
\[\psi_{r+1} = \psi_{r+1, B} + \psi_{r+1, H} + \psi_{r+1, R}.\]
We will then choose \(\phi_{r+1}, \sigma_{r+1}, \tau_{r+1}\) to eliminate these three components respectively.

\subsubsection{Base case}

The base case of \(r = 0\) is immediate, since
\[i\Lambda_k F_{\delbar_s} = i\Lambda_V F_{\delbar_0} + \order{k^{-1/2}}.\]
Indeed, the first term is exactly \(c_V\id_E\), since we assumed that \((E^0_b, h)\) is Hermite--Einstein for all \(b \in B\).

Similarly, as there are no \(k^{-1/2}\) terms, the case \(r = 1\) is automatic, and we can set \(\phi_1 = \sigma_1 = \tau_1 = 0\).

\subsubsection{Case \(r = 2\)}

The \(k^{-1}\) term is given by
\[\psi_2 = i\Lambda_HF_{\delbar_0} + i\lambda\nu_V(a).\]
Applying the decomposition from \Cref{eqn:decomposition-of-Hermitian-endomorphisms}, we may write
\[\psi_2 = \psi_{2, B} + \psi_{2, H} + \psi_{2, R}.\]

In fact,
\[\psi_{2, H} = p(i\Lambda_H(F_{\delbar_0})) - i\lambda\nu_V(a).\]
But we assumed that we have a solution to the family Hermite--Einstein equation, and so
\[\psi_{2, H} = 0.\]

In particular, we will set \(\sigma_2 = 0\). Next, we explain how to find \(\tau_2\). By the linearisation of the Hermitian--Yang--Mills equation,
\[i\Lambda_VF_{\exp(k^{-1}\tau_2)\cdot \delbar_s} = i\Lambda_V F_{\delbar_0} + k^{-1}(\Delta_{V, k}\tau_2 - i\lambda\nu_V(a)) + \order{k^{-3/2}}.\]
Here, we note that \(\Delta_{V, k}\) is the vertical Laplacian defined with respect to the Dolbeault operator \(\delbar_s\). In our case, as \(\delbar_s = \delbar_0 + \order{k^{-1/2}}\), it follows that
\[\Delta_{V, k} = \Delta_{V, 0} + \order{k^{-1/2}}.\]
As such, we may rewrite the above expansion as
\[i\Lambda_VF_{\exp(k^{-1}\tau_2)\cdot \delbar_s} = i\Lambda_V F_{\delbar_0} + k^{-1}(\Delta_{V, 0}\tau_2 - i\lambda\nu_V(a)) + \order{k^{-3/2}}.\]

\begin{lemma}
    The operator
    \[\Delta_{V, 0} \colon R \to R\]
    is invertible.
\end{lemma}

\begin{proof}
    Restricting to each \(E_b\), the statement is that the Laplacian is invertible orthogonal to its kernel, which follows from it being a self-adjoint elliptic operator. Thus, it remains to show that the solutions over \(E_b\) vary smoothly with \(b\). This follows from the fact that the \(\Delta_b\) are a smooth family of differential operators.
\end{proof}

In particular, projecting onto \(R\), we obtain that
\[\pr_R(i\Lambda_V F_{\exp(k^{-1}\tau_2)\cdot \delbar_s}) = k^{-1}(\psi_{2, R} + \Delta_V\tau_2) + \order{k^{-3/2}}.\]
Thus, we define \(\tau_2 \in R\) to be the unique solution to
\[\Delta_V\tau_2 = -\psi_{2, R}.\]

Next, we note that for any Dolbeault operator \(\delbar\) and \(\phi \in C^\infty(B)\),
\[i\Lambda_V F_{\exp(\phi) \cdot \delbar} = i\Lambda_V F_{\delbar}\]
and
\[i\Lambda_H F_{\exp(\phi) \cdot \delbar} = i\Lambda_H F_{\delbar} - \Delta_B \phi,\]
where \(\Delta_B \colon C^\infty(B) \to C^\infty(B)\) is the full Laplacian on \((B, \omega_B)\). As \(B\) is compact, \(\Delta_B\) is invertible modulo constants, and so we may find a \(\phi_2 \in C^\infty(B)\) and \(\gamma_2 \in \R\) such that
\[\Delta_B\phi_2 = -\psi_{2, B} + \gamma_2.\]

Using all of this, we then obtain that
\[i\Lambda_k F_{f_{k, 2} \cdot \delbar_s} = (\gamma_0 + \gamma_2 k^{-1})\id_E + \order{k^{-3/2}}.\]

\subsubsection{Inductive case}

Next, we will construct approximate solutions inductively, using the linearisation of the family Hermite--Einstein equation. Let
\[\psi_3 = \psi_{3, B} + \psi_{3, H} + \psi_{3, R}\]
denote the \(k^{-3/2}\) term in \(i\Lambda_k F_{f_{k, 2} \cdot \delbar_s}\). Our next goal is to understand
\[i\Lambda_k F_{\exp(k^{-1/2}\sigma_3) \cdot f_{k, 1} \cdot \delbar_s} = i\Lambda_V F_{\exp(k^{-1/2}\sigma_3) \cdot f_{k, 1} \cdot \delbar_s} + k^{-1}i\Lambda_H F_{\exp(k^{-1/2}\sigma_3) \cdot f_{k, 1} \cdot \delbar_s},\]
and we will understand these two terms separately. For the horizontal term, we obtain
\[k^{-1}i\Lambda_H F_{\exp(k^{-1/2}\sigma_3) \cdot f_{k, 1} \cdot \delbar_s} = k^{-1}i\Lambda_H F_{f_{k, 1} \cdot \delbar_s} - k^{-3/2}\Delta_{H, k, 1}\sigma_3 + \order{k^{-2}}.\]

As before, \(f_{k, 1} \cdot \delbar_s = \delbar_0 + \order{k^{-1/2}}\), and so we obtain
\[k^{-1}i\Lambda_H F_{\exp(k^{-1/2}\sigma_3) \cdot f_{k, 1} \cdot \delbar_s} = k^{-1}i\Lambda_H F_{f_{k, 1} \cdot \delbar_s} - k^{-3/2}\Delta_H\sigma_3 + \order{k^{-2}}.\]
In fact, we are only interested in the \(H_0\)-component, and so applying the projection map \(p\), we obtain
\[k^{-1}p(i\Lambda_H F_{\exp(k^{-1/2}\sigma_3) \cdot f_{k, 1} \cdot \delbar_s}) = k^{-1}p(i\Lambda_H F_{f_{k, 1} \cdot \delbar_s}) - k^{-3/2}p(\Delta_H \sigma_3) + \order{k^{-2}}.\]

For the vertical term, we first note that it suffices to calculate fibrewise.

\begin{proposition}
    \label{prop:laplacian-expansion}
    Let \(\grad_k\) be any connection on \(E_b\) such that 
    \[\grad_k = \grad_0 + k^{-1/2}A_1 + k^{-1}A_2 + \order{k^{-3/2}}\]
    and let \(\Delta_k\) be the Laplacian computed with respect to \(\grad_k\) and \(\omega_X\). 
    
    \begin{enumerate}[(i)]
        \item For any \(\sigma_1, \sigma_2 \in i\mfk_b\), we have an expansion
        \[\langle \Delta_k \sigma_1, \sigma_2 \rangle_{L^2} = k^{-1}\langle D_2 \sigma_1, \sigma_2 \rangle_{L^2} + \order{k^{-3/2}}\]
        for some fixed operator \(D_2\). 
        \item Moreover, for any \(j \ge 2\) and any \(\tau \in \mcA^0(\End(E_b))\), let \(f_k = \exp(k^{-j/2}\tau)\), and \(\widetilde\grad_k = f_k \cdot \grad_k\). Let \(\widetilde\Delta_k\) denote the corresponding Laplacian. Then for any \(\sigma_1, \sigma_2 \in i\mfk_b\), we have an expansion
        \[\langle \Delta_k\sigma_1, \sigma_2 \rangle_{L^2} = k^{-1}\langle D_2\sigma_1, \sigma_2 \rangle_{L^2} + \order{k^{-3/2}},\]
        where the operator \(D_2\) is the same as in (i).
        \item The same result as (ii) holds for \(j = 1\) and \(\tau \in i\mfk_b\).
    \end{enumerate}
\end{proposition}

\begin{proof}
    For (i), we may compute
    \begin{align*}
        \langle \Delta_k \sigma_1, \sigma_2 \rangle_{L^2} &= \langle \grad_k \sigma_1, \grad_k \sigma_2 \rangle_{L^2}, \\
        &= \langle \grad_0 \sigma_1, \grad_0 \sigma_2 \rangle_{L^2} \\
        &+ k^{-1/2}\left(\langle A_1\sigma_1, \grad_0\sigma_2 \rangle_{L^2} + \langle \grad_0\sigma_1, A_1\sigma_2 \rangle_{L^2}\right) \\
        &+ k^{-1}\left(\langle \grad_0\sigma_1, A_2 \sigma_2\rangle_{L^2} + \langle A_1\sigma_1, A_1\sigma_2\rangle_{L^2} + \langle A_2\sigma_1, \grad_0\sigma_2\rangle_{L^2}\right) \\
        &+ \order{k^{-3/2}}.
    \end{align*}
    But as we are assuming that \(\sigma_1, \sigma_2 \in i\mfk_b\), and so \(\grad_0\sigma_1 = \grad_0\sigma_2 = 0\). Most of the terms vanish, leaving us with
    \[\langle \Delta_k \sigma_1, \sigma_2\rangle_{L^2} = k^{-1}\langle A_1^*A_1\sigma_1, \sigma_2 \rangle_{L^2} + \order{k^{-3/2}}.\]

    For (ii), for \(j \ge 2\), \(j/2 \ge 1\), and \(\widetilde \grad_k = \grad_k + \order{k^{-1}}\). In particular, the \(A_1\) term is unchanged. For (iii), when \(j = 1\) and \(\tau \in i\mfk_b\), it follows that \(\grad_0\tau = 0\), which then implies that \(\widetilde\grad_k = \grad_k + \order{k^{-1}}\) as well.
\end{proof}

\begin{remark}
    A similar result is proven in \cite[Proposition 4.11]{dervanOptimalSymplecticConnections2021} in the optimal symplectic connection theory. The complication in our setting is that we need the \(\order{k^{-1}}\) term, but instead of an expansion in powers of \(k^{-1}\), we have an expansion in powers of \(k^{-1/2}\). This necessitates the more careful statement above.
\end{remark}

In particular, we may apply (i) with the Dolbeault operator \(\delbar_s = \delbar_0 + \alpha(s)\) on \(\End(E_b)\), and \(f_{k, 1}\vert_{E_b}\) satisfies the requirements in (ii). Thus, by the linearisation of the Hermitian--Yang--Mills equation,
\[p\left(i\Lambda_V F_{\exp(k^{-1/2}\sigma_3) \cdot f_{k, 1} \cdot \delbar_s}\right) = p\left(i\Lambda_V F_{f_{k, 1} \cdot \delbar_s}\right) - k^{-3/2}D_V\sigma_3 + \order{k^{-2}},\]
for some operator \(D_V\). Putting everything together, we obtain that
\[p\left(i\Lambda_k F_{\exp(k^{-1/2}\sigma_3) \cdot f_{k, 1} \cdot \delbar_s}\right) = k^{-3/2}(\psi_{3, H} - p(\Delta_H \sigma_3) - D_V\sigma_3) + \order{k^{-2}}.\]

As such, what we will now need to understand is the operator
\[\mcL = p \circ \Delta_H + D_V,\]
which we will consider as an operator
\[\mcL \colon \mcA^0(H_0) \to \mcA^0(H_0).\]

\begin{proposition}
    \label{prop:adiabatic-linearisation-elliptic}
    \(\mcL\) is a self-adjoint second order elliptic operator.
\end{proposition}

\begin{proof}
    Recall that \(H_0\) is a vector bundle over \(B\), and so the inner product on \(\mcA^0(H_0)\) is given by a fibre integral. In particular,
    \begin{align*}
        \langle \mcL \sigma, \tau \rangle_{L^2(H_0)} &= \int_B\left(\int_{X \times B/B}\left(\langle \Delta_H\sigma, \tau\rangle + \langle D_V\sigma, \tau\rangle\right)\omega_X^n\right)\omega_B^m \\
        &= \int_{X \times B} \langle \grad_H \sigma, \grad_H\tau \rangle + \langle A_{\End, V} \sigma, A_{\End, V}\tau \rangle \omega_X^n \wedge \omega_B^m,
    \end{align*}
    where \(A_{\End, V}\sigma = [-a_V^* + a_V, \sigma]\). From this expression it is clear that \(\mcL\) is self-adjoint. 
    
    For ellipticity, we note that \(D_V\) has order \(0\), and
    \[\Delta_H(f\sigma) = (\Delta_Bf)\sigma + Q(f, \dd f, \sigma, \grad_H \sigma, \grad_H^2\sigma).\]
    Thus, working in a local trivialisation, the symbol is given by the symbol of the Laplacian, hence \(\mcL\) is elliptic.
\end{proof}

Our next goal is to show that \(\mcL\) is invertible, and so we must understand the kernel of \(\mcL\). From the proof of the proposition, we see that
\[\langle \mcL\sigma, \sigma \rangle_{L^2(H_0)} = \norm{\grad_H\sigma}^2_{L^2(\End(E))} + \norm{A_{\End, V}\sigma}^2_{L^2(\End(E))}.\]
Each of the terms on the right hand side is nonnegative, and so \(\mcL(\sigma) = 0\) if and only if \(\grad_H\sigma = 0\) and \([a_V,\sigma] = 0\).

\begin{corollary}
    The operator \(\mcL\) is invertible.
\end{corollary}

\begin{proof}
    By the above computation, any elements of the kernel of \(\mcL\) must lie in the Lie algebra \(\mfg_a\), which we assumed to consist only of scalar multiples of \(\id\). Since we are assuming that \(\sigma\vert_{E_b}\) is orthogonal to \(\id\), we must then have that \(\sigma = 0\), and so \(\mcL\) is injective.
\end{proof}

\begin{remark}
    \label{rem:difference-with-ortu}
    We note a gap in \cite[Lemma 5.7]{ortuOptimalSymplecticConnections2023}, which is the corresponding statement for optimal symplectic connections. There, Ortu claims that the kernel of the linearisation is given by holomorphy potentials with respect to \(J_s\) for all \(s\). However, the proof only shows that elements of the kernel are holomorphy potentials with respect to \(J_0\) and the first order deformation. This leads us to the slightly different geometric setup in comparison with hers.
\end{remark}

With all of this in mind, we may choose \(\sigma_3\) such that
\[\mcL(\sigma_3) = \psi_{3, H},\]
and choose \(\tau_3, \phi_3\) as in the \(r = 2\) case. This then proves the \(r = 3\) case. The argument above is general, and works for any \(r \ge 3\). As such, we may construct approximate solutions, and so we have proven \Cref{prop:approximate-solutions}.

\subsection{Perturbation to a genuine solution}

\label{subsec:perturb-to-genuine-solution}

The proof of \Cref{thm:hym-total-space} will be by perturbing the approximate solutions constructed in \Cref{prop:approximate-solutions} to genuine solutions. The strategy is similar to \cite{sektnanHermitianYangMillsConnections2024}, and relies on a quantitative version of the implicit function theorem.

\begin{theorem}
    \label{thm:qual-impl-fn-thm}

    Let \(\mcR \colon V \to W\) be a differentiable map between Banach spaces. Suppose the derivative \(D\mcR\) at \(0\) is surjective, with right inverse \(\mcQ\). Let
    \begin{enumerate}[(i)]
        \item \(\delta_1\) denote the radius of the closed ball in \(V\) such that \(\mcR - D\mcR\) is \((2 \norm{\mcQ})^{-1}\)-Lipschitz,
        \item \(\delta = \delta_1/(2\norm{\mcQ})\).
    \end{enumerate}
    Then for all \(w \in W\) with \(\norm{w - \mcR(0)} < \delta\), there exists a \(v \in V\) such that \(\mcR(v) = w\).
\end{theorem}

For brevity of notation, let
\[\delbar_{k, r} = f_{k, r} \cdot \delbar_s, \quad \grad_{k, r} = f_{k, r} \cdot \grad_s.\]
As the implicit function theorem requires us to be working with Banach spaces, define
\[V_{k, d} = L^2_d(\End_{\mathrm H}(E, h), \omega_k).\]
Define maps
\begin{align*}
    \Psi_{k, r} \colon V_{k, d + 2} \times \R &\to V_{k, d} \\
    \Psi_{k, r}(\sigma, \lambda) &= i\Lambda_{\omega_k} F_{\exp(\sigma) \cdot \delbar_{k, r}} - \lambda\id_E.
\end{align*}
The linearisation of \(\Psi_{k, r}\) at \(0\) is given by
\[(\sigma, \lambda) \mapsto -\Delta_{k, r}\sigma - \lambda\id_E.\]

By assumption, \((E, \delbar_s)\) is simple for \(\abs{s} \ll 1\), and so for \(k \gg 0\), this operator is surjective. Let \(\mcP_{k, r}\) denote the right inverse to this operator. Throughout, we set \(s^2 = \lambda k^{-1}\), and assume that \(k\) is sufficiently large (resp. \(s\) is sufficiently small), so that \((E, \delbar_s)\) is simple.

\begin{proposition}
    \label{prop:operator-norm-of-right-inverse}
    For each \(d\), there exists \(C > 0\) independent of \(k\), such that the right inverse \(\mcP_{k, r} \colon V_{k, d} \to V_{k, d + 2} \times \R\) has operator norm
    \[\norm{\mcP_{k, r}}_{\mathrm{op}} \le C k.\]
\end{proposition}

The proof of this proposition will be in two parts. In \Cref{prop:poincare-inequality}, we will prove a Poincar\'e inequality, which will prove the statement for \(d = 0\). Then, we will prove a Schauder estimate in \Cref{prop:schauder-estimate}, which will allow us to prove the general case.

\begin{proposition}
    \label{prop:poincare-inequality}
    For each \(r\), there exists \(C > 0\) independent of \(k\), such that for \(k \gg 0\), and any section \(\sigma \in \mcA^0(X \times B, \End_H(E, h))\) with
    \[\int_{X \times B}\tr(\sigma)\omega_1^{m+n} = 0,\]
    we have
    \[\norm{\grad_{k, r} \sigma}^2_{L^2(\omega_k)} \ge C\norm{\sigma}^2_{L^2(\omega_k)}.\]
\end{proposition}

We follow the general strategy of \cite{sektnanHermitianYangMillsConnections2024,dervanZcriticalConnectionsBridgeland2024} of considering an orthogonal decomposition. 
A crucial difference is that in \cite{sektnanHermitianYangMillsConnections2024,dervanZcriticalConnectionsBridgeland2024}, the deformation term enters at a later order in the adiabatic limit, whereas it plays a more significant role here. 

From \cref{eqn:decomposition-of-Hermitian-endomorphisms}, we have a decomposition
\[\mcA^0(\End_H(E, h)) = \mcA^0(H) \oplus R.\]
Recall that \(H\) is a real vector bundle over \(B\), with fibre over \(b\) given by \(H_b = \mcA^0(\End_H(E_b, h)) \cap H^0(X, \End\mcE_b)\). We will prove a lower bound on the two components, and then combine them.

\begin{lemma}
    \label{lem:lower-bound-on-R}
    For each \(r\), there exists \(C > 0\) independent of \(k\), such that for \(k \gg 0\), and any section \(\sigma \in R\), we have
    \[\norm{\grad_{k, r}\sigma}^2_{L^2(\omega_k)} \ge C\norm{\sigma}^2_{L^2(\omega_k)}.\]
\end{lemma}

\begin{proof}
    As \(R_b\) is the orthogonal complement of \(\ker(\grad_{b, 0})\), for any \(\sigma \in R\), we have a pointwise over \(b \in B\) Poincar\'e inequality
    \[\norm{\grad_{b, 0}\sigma_b}^2_{L^2(\omega_X)} \ge C_b\norm{\sigma_b}^2_{L^2(\omega_X)}.\]
    A priori, \(C_b\) may not be bounded below as we vary \(b\). However, \(C_b\) can be taken to be the first non-zero eigenvalue of the Laplacian \(\Delta_{b, 0}\), and this depends continuously on \(b\) \cite[Theorem 50.16]{krieglConvenientSettingGlobal1997}. As \(B\) is compact, we have that
    \[\inf_{b \in B}C_b > 0.\]
    Thus, integrating over \(B\), we obtain
    \[\norm{\grad_{V, 0}\sigma}^2_{L^2(\omega_1)} \ge C\norm{\sigma}^2_{L^2(\omega_1)}.\]
    Note that by definition of metric \(\omega_k = \omega_X + k\omega_B\), we have the scaling relations
    \begin{align*}
        \norm{\grad_{V, 0}\sigma}^2_{L^2(\omega_k)} &= k^m\norm{\grad_{V, 0}\sigma}^2_{L^2(\omega_1)}, \\
        \norm{\grad_{H, 0}\sigma}^2_{L^2(\omega_k)} &= k^{m - 1}\norm{\grad_{H, 0}\sigma}^2_{L^2(\omega_1)}, \\
        \norm{\sigma}^2_{L^2(\omega_k)} &= k^m\norm{\sigma}^2_{L^2(\omega_1)}.
    \end{align*}
    Thus, as \(\grad_{k, r} = \grad_0 + \order{k^{-1/2}}\), we obtain that for \(k \gg 0\),
    \[\norm{\grad_{k, r}\sigma}^2_{L^2(\omega_k)} \ge C\norm{\sigma}^2_{L^2(\omega_k)}.\]

\end{proof}

\begin{lemma}
    \label{lem:lower-bound-on-H}
    For each \(r\), there exists \(C > 0\) independent of \(k\), such that for \(k \gg 0\), and any section \(\sigma \in \mcA^0(H)\) with
    \[\int_{X \times B}\tr(\sigma)\omega_1^{m+n} = 0,\]
    we have
    \[\norm{\grad_{k, r}\sigma}^2_{L^2(\omega_k)} \ge Ck^{-1}\norm{\sigma}^2_{L^2(\omega_k)}.\]
\end{lemma}

\begin{proof}
    Recall that \(\Delta_{k, r} = \Delta_{V, k, r} + k^{-1}\Delta_{H, k, r}\). Then
    \begin{align*}
        \langle \Delta_{k, r}\sigma, \sigma \rangle_{L^2(\omega_k)} &= \langle \Delta_{V, k, r}\sigma, \sigma\rangle_{L^2(\omega_k)} + k^{-1}\langle \Delta_{H, k, r}\sigma, \sigma\rangle_{L^2(\omega_k)}.
    \end{align*}
    By \Cref{prop:laplacian-expansion}, and the fact that the \(L^2(\omega_k)\) inner product on \(X \times B\) is given by integrating the \(L^2(\omega_X)\) inner product over \(B\), we obtain that for \(\sigma \in \mcA^0(H)\),
    \[\langle\Delta_{V, k, r}\sigma, \sigma\rangle_{L^2(\omega_k)} = k^{-1}\langle D_2\sigma, \sigma\rangle_{L^2(\omega_k)} + \order{k^{-3/2}\norm{\sigma}_{L^2(\omega_k)}^2},\]
    for some operator \(D_2\). On the other hand,
    \[\langle\Delta_{H, k, r}\sigma, \sigma\rangle_{L^2(\omega_k)} = \langle \Delta_{H, 0}\sigma, \sigma\rangle_{L^2(\omega_k)} + \order{k^{-1/2}\norm{\sigma}_{L^2(\omega_k)}^2}.\]
    Combining these, we see that the leading order term in the expansion of \(\langle \Delta_{k, r}\sigma, \sigma\rangle_{L^2(\omega_k)}\) is given by
    \[k^{-1}\left(\langle D_2\sigma, \sigma\rangle_{L^2(\omega_k)} + \langle \Delta_{H, 0}\sigma, \sigma\rangle_{L^2(\omega_k)}\right).\]

    Recalling the definition of the linearisation \(\mcL\) of the family Hermite--Einstein equation, this is exactly \(k^{-1}\langle \mcL\sigma, \sigma\rangle_{L^2(\omega_k)}\).
    In particular, \(\mcL \colon \mcA^0(H) \to \mcA^0(H)\) is self-adjoint, elliptic and has nonnegative eigenvalues. By our assumption that \(\mfg_a = \C \cdot \id_E\), \(\sigma\) is orthogonal to the kernel of \(\mcL\), and so there exists \(C > 0\) such that
    \[\langle \mcL\sigma, \sigma\rangle_{L^2(\omega_k)} \ge C\norm{\sigma}^2_{L^2(\omega_k)}.\]
    Thus, for \(k \gg 0\),
    \[\norm{\grad_{k, r}\sigma}^2_{L^2(\omega_k)} \ge Ck^{-1}\norm{\sigma}^2_{L^2(\omega_k)}.\]

\end{proof}

\begin{lemma}
    \label{lem:combining-bounds}
    Let \(\alpha(k), \beta(k), \gamma(k), \delta(k), \epsilon(k)\) be functions from \(\Z_{> 0}\) to \(\R\). Suppose we have that
    \begin{align*}
        \alpha(k) &\ge Ck^{-1}\delta(k), \\
        \beta(k) &\ge C\epsilon(k), \\
        \abs{\gamma(k)} &\le C_\gamma k^{-3/2}(\alpha(k) + \delta(k)) + C_\gamma k^{-1/2}(\beta(k) + \epsilon(k)), \\
        \alpha(k), \beta(k), \delta(k), \epsilon(k) &\ge 0.
    \end{align*}
    Then there exists \(C > 0\) such that for \(k \gg 0\),
    \[\alpha(k) + \beta(k) + 2\gamma(k) \ge Ck^{-1}(\delta(k) + \epsilon(k)).\]
\end{lemma}

\begin{proof}
    Using the bound on \(\abs{\gamma(k)}\), we have that
    \[\alpha(k) + \beta(k) + 2\gamma(k) \ge (1 - 2C_\gamma k^{-3/2})\alpha(k) + (1 - 2C_\gamma k^{-1/2})\beta(k) - 2C_\gamma k^{-3/2}\delta(k) - 2C_\gamma k^{-1/2}\epsilon(k).\]
    Using the lower bounds of \(\alpha(k)\) and \(\beta(k)\), we then obtain that for \(k \gg 0\),
    \begin{align*}
        \alpha(k) + \beta(k) + 2\gamma(k) &\ge (Ck^{-1} - 2C_\gamma k^{-3/2} - 2CC_\gamma k^{-5/2})\delta(k) \\
        &+ (C - 2C_\gamma k^{-1/2} - 2CC_\gamma k^{-3/2})\epsilon(k).
    \end{align*}
    Thus, for \(k \gg 0\),
    \[\alpha(k) + \beta(k) + 2\gamma(k) \ge Ck^{-1}\delta(k) + C\epsilon(k) \ge Ck^{-1}(\delta(k) + \epsilon(k))\]
    as required.
\end{proof}

\begin{proof}[Proof of \Cref{prop:poincare-inequality}]
    To conclude the result from the previous lemmas, we need to consider the mixed terms. Let \(\sigma \in \mcA^0(H)\) and \(\tau \in R\). Then we obtain that
    \[\langle \Delta_{k, r}\sigma, \tau\rangle_{L^2(\omega_k)} = k^{-1/2}\langle [A_V, \sigma], \grad_{V, 0}\tau\rangle_{L^2(\omega_k)} + \order{k^{-1}\norm{\sigma}_{L^2_1(\omega_k, \grad_0)}\norm{\tau}_{L^2_1(\omega_k, \grad_0)}}.\]
    Here, \(A_V = -a_V^* + a_V\) is the vertical first-order deformation term. Note that we have \(L^2_1\)-norms, as derivatives may appear in the remaining terms. For concreteness, we define the \(L^2_1(\omega_k, \grad_0)\)-norm by
    \begin{align*}
        \norm{\sigma}^2_{L^2_1(\omega_k, \grad_0)} &= \norm{\sigma}^2_{L^2(\omega_k)} + \norm{\grad_0\sigma}^2_{L^2(\omega_k)} \\
        &= k^m\norm{\sigma}^2_{L^2(\omega_1)} + k^m\norm{\grad_{V, 0}\sigma}^2_{L^2(\omega_1)} + k^{m - 1}\norm{\grad_{H, 0}\sigma}^2_{L^2(\omega_1)}.
    \end{align*}
    
    Now
    \[\langle[A_V, \sigma], \grad_{V, 0}\tau\rangle_{L^2(\omega_k)} = \langle \grad_{V, 0}^*[A_V, \sigma], \tau\rangle_{L^2(\omega_k)}.\]
    By the Nakano identity, applied for the operators \(\grad_{b, 0}\) on each fibre \(E_b\), it suffices to show that \(\grad_{V, 0}^{1, 0}a_V = 0\) and \(\grad_{V, 0}^{0, 1}a_V^* = 0\). As \(\grad_{b, 0}\) is the Chern connection on each \(E_b\), it is unitary and so the two conditions are equivalent. Thus, we will study \(\grad^{1,0}_{V, 0}a_V\).

    By the Maurer--Cartan equations, we have that \(\grad^{0, 1}_{V, 0}a_V = 0\). By our gauge fixing assumption, \(\left(\grad^{0, 1}_{V, 0}\right)^*a_V = 0\). Thus, \(\Delta^{0, 1}_{V, 0}a_V = 0\). But now as the central fibre \((\mcE, \delbar_0)\) of our deformation is vertically Hermite--Einstein, the vertical Laplacians are equal, and so \(\grad^{1, 0}_{V, 0}a_V = 0\). In particular, we have shown that for \(k \gg 0\),
    \[\abs{\langle\Delta_{k, r}\sigma, \tau\rangle_{L^2(\omega_k)}} \le Ck^{-1}\norm{\sigma}_{L^2_1(\omega_k, \grad_0)}\norm{\tau}_{L^2_1(\omega_k, \grad_0)} = Ck^{-3/4}\norm{\sigma}_{L^2_1(\omega_k, \grad_0)}k^{-1/4}\norm{\tau}_{L^2_1(\omega_k, \grad_0)}.\]
    Young's inequality with \(p = q = 2\) then implies that
    \[\abs{\langle\Delta_{k, r}\sigma, \tau\rangle_{L^2(\omega_k)}} \le C\left(k^{-3/2}\norm{\sigma}_{L^2_1(\omega_k, \grad_0)}^2 + k^{-1/2}\norm{\tau}_{L^2_1(\omega_k, \grad_0)}^2\right).\]

    As \(X\) and \(B\) are compact, the \(L^2_1(\omega_k, \grad_0)\) and \(L^2_1(\omega_k, \grad_{k, r})\) norms are equivalent, but the constants of equivalence may depend on \(k\). Computing explicitly, we have \(C_k\) such that
    \[C_k^{-1}\norm{\grad_0\sigma}^2_{L^2(\omega_k)} \le \norm{\grad_{k, r}\sigma}_{L^2(\omega_k)}^2 \le C_k\norm{\grad_0\sigma}_{L^2(\omega_k)}^2,\]
    where \(C_k = 2 + \order{k^{-1/2}}\). Thus, for \(k \gg 0\), we may take the same constant \(C\) for all \(k\). Thus, we obtain that for some \(C_m > 0\),
    \begin{align*}
        \abs{\langle\Delta_{k, r}\sigma, \tau\rangle_{L^2(\omega_k)}} &\le C_m\left(k^{-3/2}\norm{\sigma}_{L^2_1(\omega_k, \grad_{k, r})}^2 + k^{-1/2}\norm{\tau}_{L^2_1(\omega_k, \grad_{k, r})}^2\right) \\
        &= C_m\left(k^{-3/2}\norm{\sigma}_{L^2(\omega_k)}^2 + k^{-1/2}\norm{\tau}_{L^2(\omega_k)}^2\right) \\
        &+ C_m\left(k^{-3/2}\norm{\grad_{k, r}\sigma}_{L^2(\omega_k)}^2 + k^{-1/2}\norm{\grad_{k, r}\tau}_{L^2(\omega_k)}^2\right).
    \end{align*}
    The result then follows from \Cref{lem:lower-bound-on-R,lem:lower-bound-on-H,lem:combining-bounds}.
\end{proof}

\begin{proposition}
    \label{prop:schauder-estimate}
    There exists \(C > 0\) independent of \(k\), such that
    \[\norm{\sigma}_{L^2_{d+2}(\omega_k)} \le Ck^{-1}\left(\norm{\sigma}_{L^2(\omega_k)} + \norm{\Delta_{k, r}\sigma}_{L^2_d(\omega_k)}\right).\]
\end{proposition}

To prove this, we will first prove a result which is local on \(B\). For \(U \subseteq B\) open, define \(L^2_d(U, \omega_k)\) for the Sobolev space of sections of \(\mcE \vert_{X \times U}\).

\begin{lemma}
    Fix a point \(b \in B\), and a coodinate system \(U\) centred at \(b\). Then for every sufficiently small coordinate disc \(D_{2\rho} \subseteq U\) of radius \(2\rho\), there exists a constant \(C = C(b, U, \rho, d)\) such that for all \(\sigma \in L^2_{d+2}(D_{2\rho}, \omega_k)\),
    \[\norm{\sigma}_{L^2_{d+2}(D_{2\rho}, \omega_k)} \le C\left(\norm{\sigma}_{L^2(D_\rho, \omega_k)} + \norm{\Delta_{k, r}\sigma}_{L^2_d(D_{2\rho}, \omega_k)} + k^{-1/2}\norm{\sigma}_{L^2_{d+2}(D_{2\rho}, \omega_k)}\right)\]
    for all \(k \gg 0\).
\end{lemma}

\begin{proof}
    We first establish a bound for the operator
    \[\Delta_k = \Delta_V + k^{-1}\Delta_H,\]
    with respect to the connection \(\grad_0\). That is, we will show that there exists \(C > 0\) such that
    \begin{equation}
        \label{eqn:unperturbed-estimate}
        \norm{\sigma}_{L^2_{d+2}(D_{2\rho}, \omega_k)} \le C\left(\norm{\sigma}_{L^2(D_\rho, \omega_k)} + \norm{\Delta_k\sigma}_{L^2_d(D_{2\rho}, \omega_k)} + k^{-1/2}\norm{\sigma}_{L^2_{d+2}(D_{2\rho}, \omega_k)}\right)
    \end{equation}
    for all \(k \gg 0\). We will prove this for \(d = 0\). With respect to \(\omega_k\), the \(C^2\)-seminorm \(\abs{\dd^2\sigma}_{\omega_k}\) splits pointwise as
    \[\abs{\dd^2\sigma}^2_{\omega_k} = \abs{\dd_V^2\sigma}^2_{\omega_k} + \abs{\dd_V\dd_H\sigma}^2_{\omega_k} + \abs{\dd_H^2\sigma}^2_{\omega_k},\]
    where \(\dd_V\) denotes differentiating in the \(X\)-direction and \(\dd_H\) in the \(B\)-direction. Since \(\omega_k\) scales the \(B\) direction by \(k\), 
    \begin{align*}
        \abs{\dd_H^2\sigma}_{\omega_k} &= k^{-2}\abs{\dd_H^2\sigma}_{\omega_1}, \\
        \abs{\dd_V\dd_H\sigma}_{\omega_k} &= k^{-1}\abs{\dd_V\dd_H\sigma}_{\omega_1}. 
    \end{align*}

    Since \(\Delta_V\) restricted to each \(\End(E_b) \to X\) is elliptic, for each \(b \in U\) we obtain the estimate
    \[\left(\int_{X \times \{b\}}\abs{\dd_V^2\sigma}^2_{\omega_k}\omega_X^n\right)^{1/2} \le C_b\left(\norm{\sigma}_{L^2(X, \omega_X)} + \norm{\Delta_V\sigma}_{L^2(X, \omega_X)}\right).\]

    A priori, \(C_b\) depends on \(b\). But since \(D_\rho\) is relatively compact in \(U\), the constants of ellipticity for the \(\Delta_b\) may be uniformly bounded above and below. Since the constant \(C_b\) depends only on the constants of ellipticity and the norm of the coefficients of the operator, we may choose \(C > 0\) such that the above equality holds for all \(b \in B\). Integrating over \(D_\rho\), we obtain
    \[\left(\int_{X \times D_\rho}\abs{\dd_V^2\sigma}_{\omega_k}^2\omega_k^{m + n}\right)^{1/2} \le C\left(\norm{\sigma}_{L^2(D_\rho, \omega_k)} + \norm{\Delta_V\sigma}_{L^2(D_\rho, \omega_k)}\right).\]
    Applying the same argument to the \(C^1\)-seminorm, we obtain
    \[\left(\int_{X \times D_\rho}\abs{\dd_V\sigma}_{\omega_k}^2\omega_k^{m + n}\right)^{1/2} \le C\left(\norm{\sigma}_{L^2(D_\rho, \omega_k)} + \norm{\Delta_V\sigma}_{L^2(D_\rho, \omega_k)}\right).\]

    Next, we estimate the horizontal term. Since \(\Delta_H\) is elliptic on \(U\), for all \(x \in X\) we obtain an estimate
    \[\norm{\sigma}_{L^2_2(\{x\} \times D_\rho, \omega_B)} \le C_x\left(\norm{\sigma}_{L^2(\{x\} \times D_{2\rho}, \omega_B)} + \norm{\Delta_H\sigma}_{L^2(\{x\} \times D_{2\rho}, \omega_B)}\right).\]
    The same argument as above implies that we may choose \(C_x\) independent of \(x \in X\). Using the scaling properties, we obtain that
    \begin{align*}
        \left(\int_{X \times D_\rho}\abs{\dd_H^2\sigma}^2_{\omega_k}\omega_k^{m+n}\right)^{1/2} &\le Ck^{-2}\left(\norm{\sigma}_{L^2(D_{2\rho}, \omega_k)} + \norm{\Delta_H\sigma}_{L^2(D_{2\rho}, \omega_k)}\right), \\
        &\le Ck^{-1}\left(\norm{\sigma}_{L^2(D_{2\rho}, \omega_k)} + \norm{k^{-1}\Delta_H\sigma}_{L^2(D_{2\rho}, \omega_k)}\right)
    \end{align*}
    for \(k \gg 0\). Similarly, as the horizontal \(C^1\)-seminorm pointwise scales like \(k^{-1}\), we obtain
   \begin{align*}
        \left(\int_{X \times D_\rho}\abs{\dd_H\sigma}^2_{\omega_k}\omega_k^{m+n}\right)^{1/2} &\le Ck^{-1}\left(\norm{\sigma}_{L^2(D_{2\rho}, \omega_k)} + \norm{\Delta_H\sigma}_{L^2(D_{2\rho}, \omega_k)}\right), \\
        &\le C\left(\norm{\sigma}_{L^2(D_{2\rho}, \omega_k)} + \norm{k^{-1}\Delta_H\sigma}_{L^2(D_{2\rho}, \omega_k)}\right).
    \end{align*}

    Finally, we will bound the mixed second derivative term. Since \(\Delta_1 = \Delta_V + \Delta_H\) is elliptic on \(X \times U\), there exists \(C > 0\) such that
    \[\left(\int_{X \times D_\rho}\abs{\dd_V\dd_H\sigma}^2_{\omega_1}\omega_1^{m + n}\right) \le C \left(\norm{\sigma}_{L^2(D_{2\rho}, \omega_1)} + \norm{\Delta_1\sigma}_{L^2(D_{2\rho}, \omega_1)}\right).\]
    By the scaling properties, this implies that
    \[\left(\int_{X \times D_\rho}\abs{\dd_V\dd_H\sigma}^2_{\omega_k}\omega_k^{m + n}\right) \le Ck^{-1}\left(\norm{\sigma}_{L^2(D_{2\rho}, \omega_k)} + \norm{\Delta_1\sigma}_{L^2(D_{2\rho}, \omega_k)}\right).\]
    Now for \(k \ge 1\),
    \begin{align*}
        \norm{\Delta_k\sigma}^2_{L^2(D_{2\rho}, \omega_k)} - \norm{k^{-1}\Delta_1\sigma}^2_{L^2(D_{2\rho}, \omega_k)}
        &= \left(1 - \frac{1}{k^2}\right)\norm{\Delta_V(\sigma)}^2_{L^2(D_{2\rho}, \omega_k)} \\ 
        &+ 2\left(\frac{1}{k} - \frac{1}{k^2}\right)\langle \Delta_V\sigma, \Delta_H\sigma\rangle_{L^2(D_{2\rho}, \omega_k)}, \\
        &\ge -2\left(\frac{1}{k} - \frac{1}{k^2}\right)\abs{\langle \Delta_V\sigma, \Delta_H\sigma \rangle_{L^2(D_{2\rho}, \omega_k)}}, \\
        &\ge -2\left(\frac{1}{k} - \frac{1}{k^2}\right)\norm{\Delta_V\sigma}_{L^2(D_{2\rho}, \omega_k)}\norm{\Delta_H\sigma}_{L^2(D_{2\rho}, \omega_k)}.
    \end{align*}
    But \(\Delta_V, \Delta_H \colon L^2_2 \to L^2\) are bounded operators, and so we obtain that
    \[\norm{k^{-1}\Delta_1\sigma}_{L^2(D_{2\rho}, \omega_k)} \le \norm{\Delta_k\sigma}_{L^2(D_{2\rho}, \omega_k)} + C k^{-1/2}\norm{\sigma}_{L^2_2(D_{2\rho}, \omega_k)}.\]
    Thus, we have some \(C > 0\) such that
    \[\left(\int_{X \times D_{\rho}}\abs{\dd_V\dd_H\sigma}^2_{\omega_k}\omega_k^{m + n}\right)^{1/2} \le C\left(\norm{\sigma}_{L^2_2(D_{2\rho}, \omega_k)} + \norm{\Delta_k\sigma}_{L^2(D_{2\rho}, \omega_k)} + k^{-1/2}\norm{\sigma}_{L^2_{d+2}(D_{2\rho}, \omega_k)}\right).\]

    \Cref{eqn:unperturbed-estimate} follows by combining the estimates which we have proven. To conclude, we note that \(\delbar_{k, r}\) is an \(\order{k^{-1/2}}\) perturbation of \(\delbar_k\), and so we obtain an estimate of the form
    \[\norm{\Delta_{k, r}\sigma - \Delta_k\sigma}_{L^2_d(D_{2\rho}, \omega_k)} \le Ck^{-1/2}\norm{\sigma}_{L^2_{d + 2}(D_{2\rho}, \omega_k)}.\]
\end{proof}

With this lemma in hand, the proof of \Cref{prop:schauder-estimate} follows by an identical proof as \cite[Proposition 27]{sektnanHermitianYangMillsConnections2024}, and the proof of \Cref{prop:operator-norm-of-right-inverse} follows as in the proof of \cite[Proposition 25]{sektnanHermitianYangMillsConnections2024}. To conclude, we will need a Lipschitz bound for the non-linear term \(\mcN_{k, r} = \Psi_{k, r} - \dd\Psi_{k, r}\).

\begin{lemma}
    \label{lem:non-linear-part}
    There exists \(c, C > 0\) such that for \(k \gg 0\), if \(s_1, s_2 \in V_{k, d + 2}\) are such that \(\norm{s_i}_{L^2_{d + 2}} \le c\), then
    \[\norm{\mcN_{k,r}(s_1) - \mcN_{k, r}(s_2)}_{L^2_d} \le C\left(\norm{s_1}_{L^2_{d+2}} + \norm{s_2}_{L^2_{d+2}}\right)\norm{s_1 - s_2}_{L^2_{d+2}}.\]
\end{lemma}

The proof of this lemma uses the mean value inequality, and is identical to \cite[Lemma 29]{sektnanHermitianYangMillsConnections2024}. Finally, we need to show that the estimates we have proven also hold in the various Sobolev spaces.

\begin{lemma}
    [{\cite[Lemma 30]{sektnanHermitianYangMillsConnections2024}}]
    \label{lem:norm-bounds}
    The approximate solutions satisfy
    \[\norm{i\Lambda_{\omega_k}F_{\delbar_{k, r}} - \sum_{j=0}^rk^{-j/2}\gamma_j\id_E}_{C^d} = \order{k^{-(r + 2)/2}},\]
    and
    \[\norm{i\Lambda_{\omega_k}F_{\delbar_{k, r}} - \sum_{j=0}^rk^{-j/2}\gamma_j\id_E}_{L^2_d} = \order{k^{-(r + 1)/2}}.\]
\end{lemma}

With all of this, we can now prove \Cref{thm:hym-total-space}.

\begin{proof}
    [Proof of \Cref{thm:hym-total-space}]

    We wish to show that when \(r\) is sufficiently large, then \(\Psi_{k, r}\) has a zero for \(k \gg 0\). \Cref{lem:non-linear-part} implies that there exists a constant \(c > 0\) such that for all \(\rho > 0\) sufficiently small and \(k \gg 0\), \(\mcN_{k, r}\) is \(c\rho\)-Lipschitz on the ball of radius \(\rho\). By \Cref{prop:operator-norm-of-right-inverse}, \((2\norm{\mcP_{k, r}})^{-1}\) is bounded below by \(Ck^{-1}\) for some \(C > 0\). Combining these two, we obtain that there is \(C' > 0\) such that the radius \(\delta_1\) on which \(\mcN_{j, k}\) is \((2\norm{\mcP_{j, k}})^{-1}\)-Lipschitz satisfies
    \[\delta_1 \ge C' k^{-1}.\]
    Combining this with the bound for \(\norm{\mcP_{j, k}}\), we obtain that there exists a \(C'' > 0\) such that the \(\delta\) in \Cref{thm:qual-impl-fn-thm} satisfies
    \[\delta \ge C''k^{-2}.\]
    Thus, we can use \Cref{thm:qual-impl-fn-thm} to find a root of \(\Psi_{k, r}\), provided that \(\norm{\Psi_{k, r}(0)} \le C''k^{-2}\). By \Cref{lem:norm-bounds}, this holds for all \(k \gg 0\) provided \(r \ge 4\). Thus, we can find a zero in \(L^2_{d+2}\). Elliptic regularity implies that we in fact have a smooth solution.
\end{proof}

\section{The family Hermite--Einstein flow}

\label{sec:family-Hermite--Einstein-flow}

In the study of the Hermite--Einstein equations, one of the most important tools is the Hermite--Einstein flow, which is the parabolic PDE on Hermitian metrics, given by
\begin{equation}
    \label{eqn:Hermite--Einstein-flow}
    h^{-1}\pdv{h}{t} = -2\left(i\Lambda_\omega F_h - c\id_E\right).
\end{equation}

For us, it will be more convenient to write this in terms of endomorphisms. Fix a Hermitian metric \(h_0\) on \(E\), and write any other Hermitian metric as \(h = h_0\sigma\), where \(\sigma\) is a positive definite, self-adjoint endomorphism with respect to \(h_0\). Then the Hermite--Einstein flow may be written as the evolution equation on endomorphisms

\begin{equation}
    \label{eqn:Hermite--Einstein-flow-endomorphism}
    \pdv{\sigma}{t} = -2\sigma\left(i\Lambda_\omega F_{h_0\sigma} - c\id_E\right).
\end{equation}

The Hermite--Einstein flow was introduced by Donaldson \cite{donaldsonSelfdualYangMillsConnections1985}, who proved the long-time existence of the flow. Using this, he proved the existence of Hermite--Einstein metrics on stable bundles over projective surfaces. The heat flow approach has been extended to other settings, in particular, to Higgs bundles and to manifolds with boundary by Simpson \cite{simpsonConstructingVariationsHodge1988}, and to reflexive sheaves by Bando--Siu \cite{bandoStableSheavesEinsteinHermitian1994}. 

Even in the unstable case, the flow plays a key role. In \cite{daskalopoulosConvergencePropertiesYangMills2004,daskalopoulosBlowupSetYangMills2007}, Daskalopoulos--Wentworth shows that for K\"ahler surfaces, the corresponding connections converge away from an analytic bubbling set and along a subsequence, to the graded object of the Harder--Narasimhan--Seshadri filtration. This was generalised by Jacob \cite{jacobLimitYangMillsFlow2015,jacobYangMillsFlowAtiyahBott2016} and Sibley \cite{sibleyAsymptoticsYangMillsFlow2015} to higher dimensions.

In this section, we will define a flow associated to the family Hermite--Einstein equation, and prove long-time existence of the flow. In \Cref{subsec:dirichlet-problem}, we use the long time existence of the flow to prove the existence of solutions to the Dirichlet problem for the family Hermite--Einstein equation.

Let \((X, \omega_X)\) and \((B, \omega_B)\) be compact K\"ahler manifolds. Let \(\mcE \to X \times B\) be a holomorphic vector bundle, and fix a Hermitian metric \(h\) on \(\mcE\). Suppose that \((\mcE_b, h)\) is Hermite--Einstein for all \(b \in B\). Finally, let \(a \in \mcA^0(T^*X^{0, 1} \otimes \End E)\) be a first order vertical deformation of \(\delbar_{\mcE}\).

\begin{definition}
    The \emph{family Hermite--Einstein flow} is the evolution equation
    \begin{equation}
        \label{eqn:family-Hermite--Einstein-flow}
        \pdv{\sigma}{t} = -2\sigma\left(p_{h\sigma}(i\Lambda_H F_{h\sigma}) - \lambda i\nu_{h\sigma}(a)\right),
    \end{equation}
    where \(\sigma(t)\) is a path in \(\mcA^0(H_0)\). In other words, \(h\sigma\) is a fibrewise Hermite--Einstein metric on \(\mcE\).
\end{definition}

Define
\[P_\lambda(\sigma) = p_{h\sigma}(i\Lambda_H F_{h\sigma}) - \lambda i\nu_{h\sigma}(a),\]
so that the family Hermite--Einstein flow may be written as
\[\pdv{\sigma}{t} = -2\sigma P_\lambda(\sigma).\]
We will be considering \(\lambda > 0\) as fixed, and so we will write \(P = P_\lambda\).

Our goal will be to prove that the solution to the family Hermite--Einstein flow exists for all time, by adapting the arguments of Donaldson and Simpson \cite{donaldsonSelfdualYangMillsConnections1985,simpsonConstructingVariationsHodge1988}, using the analogy between the family Hermite--Einstein equation and the Hermite--Einstein equation outlined in \Cref{sec:space-of-vertically-he-metrics}. The main novelty is \Cref{lem:moment-map-evolution}, which describes how the deformation term evolves in time, using the moment map property.

Recall \Cref{assumption:constant-dimension}, which states that
\[\dim(H^0(X, \End(\mcE_b))) \text{ is constant,}\]
and which will be assumed throughout.

For short time existence, we will appeal to the general theory of parabolic equations.

\begin{lemma}
    The linearisation of \(P\) at \(\sigma = \id_E\) is elliptic.
\end{lemma}

\begin{proof}
    The linearisation of \(P\) at \(\sigma = \id_E\) is given by
    \[L(\tau) = -2p_h(\Delta_H^{1, 0}\tau) + Q(\tau)\]
    where \(Q(\tau)\) is an order-zero operator, given by linearising the projection operator \(p_{h\sigma}\) and the deformation term \(i\nu_{h\sigma}(a)\). The same argument as in \Cref{prop:adiabatic-linearisation-elliptic} shows that the leading order term is elliptic.
\end{proof}

\begin{corollary}
    The linearisation of the family Hermite--Einstein flow is parabolic. In particular, solutions exist for short time.
\end{corollary}

Define the function \(\theta = \theta(t, b)\) by
\[\theta = \abs{P(\sigma)}_{(h\sigma)_\mcF}^2 = \int_{X \times B/B}\tr(P(\sigma)^2)\omega_X^n.\]

\begin{proposition}
    \label{prop:subsolution-to-heat-equation}
    Suppose \(\sigma = \sigma(t)\) is a solution to the family Hermite--Einstein flow. Then \(\theta\) is a subsolution to the heat equation on \(B\):
    \[\pdv{\theta}{t} + \Delta_B^{1, 0}\theta \le 0.\]
\end{proposition}

To prove this, we will need to understand how the various terms evolve. The main novelty compared to results for the Hermite--Einstein flow is \Cref{lem:moment-map-evolution}, which describes how the deformation term evolves in time.

\begin{lemma}
    \label{lem:curvature-evolution}
    Suppose \(\sigma = \sigma(t)\) is a solution to the family Hermite--Einstein flow. Then
    \[\pdv{F_{h\sigma}}{t} = -2\grad^{0, 1}\grad^{1, 0}_{h\sigma}P(\sigma).\]
\end{lemma}

\begin{proof}
    We recall that
    \[F_{h\sigma} = F_h + \grad^{0, 1}\left(\sigma^{-1}\grad_h^{1, 0}\sigma\right).\]
    Thus,
    \[\pdv{F_{h\sigma}}{t} = \grad^{0, 1}\left(-\sigma^{-1}\pdv{\sigma}{t}\sigma^{-1}\grad_h^{1, 0}\sigma + \sigma^{-1}\grad_h^{1, 0}\pdv{\sigma}{t}\right).\]
    Substituting in the right hand side of \Cref{eqn:family-Hermite--Einstein-flow}, we then obtain
    \begin{align*}
        \pdv{F_{h\sigma}}{t} &= 2\grad^{0, 1}\left(P(\sigma)\sigma^{-1}\grad_h^{1, 0}\sigma - \sigma^{-1}\grad_h^{1,0}\left(\sigma P(\sigma)\right)\right), \\
        &= 2\grad^{0, 1}\left(P(\sigma)\sigma^{-1}\grad_h^{1,0}\sigma - \sigma^{-1}(\grad_h^{1, 0}\sigma)P(\sigma) - \grad_h^{1, 0}P(\sigma)\right).
    \end{align*}
    Now we recall that on sections of \(E\),
    \[\grad^{1, 0}_{h\sigma} = \grad^{1, 0}_h + \sigma^{-1}\grad_h^{1,0}\sigma.\]
    Using this formula, we obtain the corresponding result.
\end{proof}

\begin{lemma}
    \label{lem:projected-curvature-evolution}
    Suppose \(\sigma = \sigma(t)\) is a solution to the family Hermite--Einstein flow. Then
    \[p_{h\sigma}\left(\pdv{t}P(\sigma)\right) = -2p_{h\sigma}\left(i\Lambda_H\grad^{0, 1}\grad^{1, 0}_{h\sigma}P(\sigma)\right)\]
\end{lemma}

\begin{proof}
    When we compute the derivative, we obtain two terms, one from the variation of \(p_{h\sigma}\), and the other from the variation of \(i\Lambda_H F_{h\sigma}\). For the first term, we will need to understand how \(p_{h\sigma}\) depends on \(\sigma\). For any \(\alpha\) which is \(h\sigma\)-Hermitian, 
    \[p_{h\sigma}(\alpha) = \sigma^{-1/2}p_h\left(\sigma^{1/2}\alpha \sigma^{-1/2}\right)\sigma^{1/2}.\]
    Let \(\tau = \sigma^{1/2}\). By changing our reference metric to \(h\sigma(t_0)\), we may assume without loss of generality that \(\sigma(t_0) = \id_E\). Now
    \[p_\tau(\alpha) = \tau^{-1}p_h(\tau\alpha \tau^{-1})\tau.\]
    The derivative of this, at \(\tau = \id_E\), is given by
    \[-\pdv{\tau}{t}p_h(\alpha) + p_h\left(\pdv{\tau}{t}\alpha\right) - p_h\left(\alpha\pdv{\tau}{t}\right) + p_h(\alpha)\pdv{\tau}{t} = p_h\left(\left[\pdv{\tau}{t}, \alpha\right]\right) + \left[p_h(\alpha), \pdv{\tau}{t}\right].\]
    Now as the commutator of Hermitian endomorphisms is skew-Hermitian, it follows that
    \[p_h\left(\left[\pdv{\tau}{t}, \alpha\right]\right) = 0\]
    and
    \[p_h\left(\left[p_h(\alpha), \pdv{\tau}{t}\right]\right) = 0.\]

    Thus, the only contribution is from the variation of \(i\Lambda_H F_{h\sigma}\). By \Cref{lem:curvature-evolution}, and that \(p_{h\sigma}^2 = p_{h\sigma}\), we obtain that
    \[p_{h\sigma}\left(\pdv{t}P(\sigma)\right) = -2p_{h\sigma}\left(i\Lambda_H\grad^{0, 1}\grad^{1, 0}_{h\sigma}P(\sigma)\right)\]
    as required.
\end{proof}

Next, we will need to understand how the deformation term evolves. For this, we will use the moment map interpretation of \(\nu_h\).

\begin{lemma}
    \label{lem:moment-map-evolution}
    Suppose \(\sigma = \sigma(t)\) is a solution to the family Hermite--Einstein flow. Then
    \[\Re\int_{X \times B/B}\tr(P(\sigma) \cdot i\pdv{\nu}{t})\omega_X^n = \abs{[P(\sigma), a]}^2.\]
\end{lemma}

\begin{proof}
    Recall that
    \[p_{h\sigma}(i\Lambda_VF_{h\sigma, \delbar_s}) = -s^2\nu_{h\sigma}(a) + \order{s^3}.\]
    We also have that
    \[F_{h\sigma, \delbar_s} = \sigma^{-1/2} \circ F_{h, \sigma^{1/2}\cdot\delbar_s} \circ \sigma^{1/2},\]
    and so
    \[p_{h\sigma}(i\Lambda_VF_{h\sigma, \delbar_s}) = \sigma^{-1/2}p_h(i\Lambda_V F_{h, \sigma^{1/2}\cdot\delbar_s})\sigma^{1/2}.\]
    Next, since \(\sigma^{1/2}\) is \(\delbar_0\)-holomorphic,
    \[\sigma^{1/2}\cdot\delbar_s = \delbar_0 + \sigma^{1/2}\circ\alpha(s)\circ\sigma^{-1/2}\]
    and so
    \[\nu_{h\sigma}(a) = \sigma^{-1/2}\nu_h(\sigma^{1/2}a\sigma^{-1/2})\sigma^{1/2}.\]
    By changing our reference metric, we may assume that at \(t = t_0\), \(\sigma(t_0) = \id\). In this case,
    \[\pdv{\nu}{t} = -\frac{1}{2}\left[\pdv{\sigma}{t}, \nu_h(a)\right] + \frac{1}{2}(\dd\nu_h)_a\left(\left[\pdv{\sigma}{t}, a\right]\right).\]
    As we assumed \(\sigma(t_0) = \id\),
    \[\pdv{\sigma}{t} = -2P,\]
    where \(P = P(\id)\). Thus, we obtain
    \[\pdv{\nu}{t} = [P, \nu_h(a)] + -(\dd\nu_h)_a([P, a]).\]
    Since \(P\) is Hermitian and \(i\nu_h(a)\) is Hermitian, \([P, i\nu_h(a)]\) is skew-Hermitian. As Hermitian and skew-Hermitian endomorphisms are \(L^2(\Re h)\)-orthogonal, we may ignore the first term. Thus, what we wish to compute is
    \[-\int_{X \times \{b\}}\tr(iP \cdot (\dd\nu_h)_a([P, a]))\omega_X^n.\]
    For this, we will use the fact that \(\nu_h\) is a moment map, to obtain that
    \begin{align*}
        -\int_{X \times \{b\}}\tr(iP \cdot (\dd\nu_h)_a([P, a]))\omega_X^n &= \Omega([P, a], [iP, a]), \\
        &= \Omega([P, a], i[P, a]), \\
        &= \abs{[P, a]}^2.
    \end{align*}
    
\end{proof}

\begin{proof}
    [Proof of \Cref{prop:subsolution-to-heat-equation}] Using \Cref{lem:projected-curvature-evolution} and \Cref{lem:moment-map-evolution},
    \begin{align*}
        \pdv{\theta}{t} &= 2\int_{X \times B/B}\tr(P(\sigma) \cdot p_{h\sigma}\left(-i\Lambda_H\left(\grad^{0, 1}\grad^{1, 0}_{h\sigma}P(\sigma)\right)\right))\omega_X^n  \\
        &-2\lambda\Re\int_{X \times B/B}\tr(P(\sigma) \cdot i\pdv{\nu}{t})\omega_X^n, \\
        &= -2i\int_{X \times B/B}\tr(P(\sigma) \cdot \Lambda_H\left(\grad^{0, 1}\grad^{1, 0}_{h\sigma}P(\sigma)\right))\omega_X^n -2\lambda\abs{[P(\sigma), a]}^2.
    \end{align*}

    To compute the Laplacian, we will need to view \(P(\sigma)\) as a section of \(\mcF\), and \(\theta\) as the \((h\sigma)_\mcF\)-norm-squared of this section. Thus,
    \begin{align*}
        \del_B\theta &= 2\int_{X \times B/B}\tr(P(\sigma) \cdot \grad^{0, 1}_{\mcF, h\sigma}P(\sigma))\omega_X^n, \\
        &= 2\int_{X \times B/B}\tr(P(\sigma) \cdot \grad^{0, 1}_{H, h\sigma}P(\sigma))\omega_X^n.
    \end{align*}
    Similarly, 
    \begin{align*}
        \delbar_B\del_B\theta &= 2\int_{X \times B/B}\tr(P(\sigma) \cdot \grad^{0, 1}_H\grad^{1, 0}_{H, h\sigma}P(\sigma))\omega_X^n, \\
        &+ 2\int_{X \times B/B}(h\sigma)(\pi_{h\sigma}(\grad^{0, 1}_H P(\sigma)), \pi_{h\sigma}(\grad^{1, 0}_{H, h\sigma}P(\sigma)))\omega_X^n.
    \end{align*}

    By the K\"ahler identities,
    \[\Delta_B^{1, 0} = \del_B^*\del_B = i\Lambda_B\delbar_B\del_B,\]
    and so combining all of the above and, we obtain that
    \[\pdv{\theta}{t} + \Delta_B^{1, 0}\theta = -2\abs{\grad^{1, 0}_{\mcF, h\sigma}P(\sigma)}_{(h\sigma)_\mcF}^2 - 2\lambda\abs{[P(\sigma), a]}^2 \le 0\]
    as required.
\end{proof}

\begin{corollary}
    \(\sup_B \abs{P(\sigma)}^2_{(h\sigma)_\mcF}\) is non-increasing with time.
\end{corollary}

\begin{proof}
    This follows by the maximum principle for the heat equation.
\end{proof}

\begin{proposition}
    Solutions to the flow are unique.
\end{proposition}

\begin{proof}
    Suppose \(\sigma, \tau\) are solutions. Define
    \begin{align*}
        \beta_{\sigma, \tau}(t) &= \int_{X \times B/B}\tr(\sigma^{-1}\tau)\omega_X^n.
    \end{align*}
    We will first show that \(\beta_{\sigma, \tau}(t)\) satisfies
    \begin{equation}
        \label{eqn:flow-uniqueness-aux-heat}
        \pdv{\beta_{\sigma, \tau}}{t} + 2\Delta_B^{1, 0}\beta_{\sigma, \tau} \le 0.
    \end{equation}
    
    First of all,
    \begin{align*}
        \pdv{\beta_{\sigma, \tau}}{t} &= \int_{X \times B/B}\tr(-\sigma^{-1}\pdv{\sigma}{t}\sigma^{-1}\tau + \sigma^{-1}\pdv{\tau}{t})\omega_X^n \\
        &= -2\int_{X \times B/B}\tr(\sigma^{-1}\tau \left(p_{h\tau}(i\Lambda_H F_{h\tau}) - p_{h\sigma}(i\Lambda_H F_{h\sigma})\right))\omega_X^n \\
        &+2\lambda\int_{X \times B/B}\tr(\sigma^{-1}\tau\left(i\nu_{h\tau}(a) - i\nu_{h\sigma}(a)\right))\omega_X^n.
    \end{align*}
    Next, note that for any \(\alpha\) which is \(h\tau\)-Hermitian,
    \[p_{h\tau}(\alpha) = (\sigma^{-1}\tau)^{-1/2}p_{h\sigma}\left((\sigma^{-1}\tau)^{1/2}\alpha (\sigma^{-1}\tau)^{-1/2}\right)(\sigma^{-1}\tau)^{1/2}.\]
    The same expression applies to \(\pi_{h\tau}\) as well. Thus, since \(\sigma^{-1}\tau\) is vertically holomorphic, we obtain that
    \begin{align*}
        \pdv{\beta_{\sigma, \tau}}{t} &= -2\int_{X \times B/B}\tr(\sigma^{-1}\tau\left(i\Lambda_H F_{h\tau} - i\Lambda_H F_{h\sigma}\right))\omega_X^n \\
        &+ 2\lambda\int_{X \times B/B}\tr(\sigma^{-1}\tau\left(i\nu_{h\tau}(a) - i\nu_{h\sigma}(a)\right))\omega_X^n.
    \end{align*}
    We will now understand the two terms separately.

    Set \(\gamma = \sigma^{-1}\tau\). Then
    \[F_{h\tau} = F_{h\sigma} + \grad^{0, 1}(\gamma^{-1}\grad^{1, 0}_{h\sigma}\gamma) = F_{h\sigma} + \gamma^{-1}\left(\grad^{0, 1}\grad^{1, 0}_{h\sigma}\gamma - (\grad^{0, 1}\gamma)\gamma^{-1}\grad^{1, 0}_{h\sigma}\gamma\right).\]
    Thus,
    \[i\Lambda_H(F_{h\tau} - F_{h\sigma}) = \gamma^{-1}\Delta^{1, 0}_{H, h\sigma}\gamma - i\gamma^{-1}\Lambda_H((\grad^{0, 1}\gamma) \gamma^{-1}\grad^{1, 0}_{h\sigma}\gamma).\]
    Substituting this in, we obtain that
    \begin{align*}
        \pdv{\beta_{\sigma, \tau}}{t} &= -2\int_{X \times B/B}\tr(\Delta_{H, h\sigma}^{1, 0}\gamma) + 2i\int_{X \times B/B}\Lambda_H \tr(\gamma^{-1}(\grad^{0, 1}\gamma)\gamma^{-1}\grad^{1, 0}_{h\sigma}\gamma)\omega_X^n \\
        &+ 2\lambda\int_{X \times B/B}\tr(\sigma^{-1}\tau(i\nu_{h\tau}(a) - i\nu_{h\sigma}(a)))\omega_X^n.
    \end{align*}
    The first term is exactly \(-2\Delta_B^{1, 0}\beta_{\sigma, \tau}\), and the second term is non-positive. Thus, what we need to understand is the final term.

    Recall that
    \[p_{h\sigma}(i\Lambda_V F_{h\sigma, \delbar_s}) = -s^2i\nu_{h\sigma}(a) + \order{s^3}.\]
    Thus, we obtain that
    \begin{align*}
        \int_{X \times B/B}\tr(\sigma^{-1}\tau(i\Lambda_VF_{h\tau, s} - i\Lambda_V F_{h\sigma, \delbar_s}))\omega_X^n &= -s^2\int_{X \times B/B}\tr(\sigma^{-1}\tau(i\nu_{h\tau}(a) - i\nu_{h\sigma}(a)))\omega_X^n \\
        &+ \order{s^3}.
    \end{align*}
    Setting \(\gamma = \sigma^{-1}\tau\), and using the same reasoning as above, we obtain that
    \begin{equation}
        \label{eqn:vertical-term-expansion}
        \begin{split}
            \int_{X \times B/B}\tr(\gamma(i\Lambda_V F_{h\tau, s} - i\Lambda_V F_{h\sigma, s}))\omega_X^n &= \int_{X \times B/B}\tr(\Delta^{1, 0}_{V, h\sigma, s}\gamma)\omega_X^n \\
            &- i\int_{X \times B/B}\Lambda_V\tr(\gamma^{-1}(\grad^{0, 1}_s\gamma)\gamma^{-1}\grad^{1, 0}_{h\sigma, s}\gamma)\omega_X^n.
        \end{split}
    \end{equation}
    For the first term, using the fact that the vertical Laplacian is self-adjoint, we have that for any \(b \in B\),
    \begin{align*}
        \int_{X \times \{b\}}\tr(\Delta^{1, 0}_{b, h\sigma, s}\gamma)\omega_X^n &= \int_{X \times \{b\}}\tr(\Delta^{1, 0}_{b, h\sigma, s}\gamma \circ \id)\omega_X^n \\
        &= \int_{X \times \{b\}}\tr(\gamma \circ \Delta^{1, 0}_{b, h\sigma, s}\id)\omega_X^n \\
        &= 0.
    \end{align*}
    The same logic as before shows that the second term on the right hand side of \Cref{eqn:vertical-term-expansion} is nonnegative, and so the left hand side is nonnegative. As the pointwise limit of nonnegative functions is nonnegative, by taking the limit \(s \to 0\), we see that
    \[\int_{X \times B/B}\tr(\sigma^{-1}\tau(i\nu_{h\tau}(a) - i\nu_{h\sigma}(a)))\omega_X^n \le 0.\]
    Putting everything together shows inequality (\ref{eqn:flow-uniqueness-aux-heat}).

    Now notice that
    \[\tr(\sigma^{-1}\tau) + \tr(\tau^{-1}\sigma) - 2\rank(E)\]
    is a nonnegative function, and is zero if and only if \(\sigma = \tau\). But
    \begin{equation}
        \label{eqn:flow-aux-distance}
        \eta(\sigma, \tau) = \beta_{\sigma, \tau} + \beta_{\tau, \sigma} - 2\vol(X, \omega_X)\rank(E)
    \end{equation}
    satisfies (a rescaled version of) the heat equation, and so by the maximum principle, it is non-increasing. Thus, if \(\sigma(0) = \tau(0)\), then \(\sigma(t) = \tau(t)\) for all \(t\), and so we have shown that the solution to the family Hermite--Einstein flow is unique.
\end{proof}

\begin{corollary}
    Suppose a smooth solution \(\sigma(t)\) exists on \([0, T)\) for some \(T > 0\). Then \(\sigma(t)\) converges in \(C^0\) to a continuous limit \(\sigma_T\) as \(t \to T\).
\end{corollary}

\begin{remark}
    Here and below, \(C^k\)-norms of \(\sigma\) refers to \(C^k\)-norms as sections of \(\mcF \to B\). In particular, by definition \(\sigma_T\) will be smooth in the \(X\)-directions, and the claim is that it is continuous in the \(B\)-directions.
\end{remark}

\begin{proof}
    The space of fibrewise Hermite--Einstein metrics embeds into the space of all Hermitian metrics as a totally geodesic submanifold. In particular, one proof follows by directly appealing to the result in \cite[Corollary 15]{donaldsonSelfdualYangMillsConnections1985} for the Hermite--Einstein flow. The proof below instead adapts the proof of \cite[Corollary 15]{donaldsonSelfdualYangMillsConnections1985}.

    Recall from \Cref{subsec:space-of-vertically-he-metrics} that the space of fibrewise Hermite--Einstein metrics is the space of sections of a fibre bundle. Over \(b \in B\), we may identify the fibre with the homogeneous space
    \[\frac{\Aut(\mcE_b)}{\Aut(\mcE_b) \cap U(E_b, h)}.\]
    In particular, we have a complete Riemannian metric \(d\) on this space. This metric compares uniformly with \(\eta\) as defined in \Cref{eqn:flow-aux-distance}. That is, there exists monotone functions \(f_1, f_2\) such that \(\eta \le f_1(d)\) and \(d \le f_2(\eta)\). This follows from the corresponding results for Hermitian inner products on vector spaces, as used in \cite[Corollary 15]{donaldsonSelfdualYangMillsConnections1985}, combined with \Cref{lem:metric-on-homogeneous-spaces}.
    
    It then suffices to show that given \(\epsilon > 0\), we can find \(\delta > 0\) such that
    \[\sup_B \eta(\sigma(t), \sigma(t')) < \epsilon\]
    for all \(t, t' > T - \delta\), since this implies that the \(\sigma(t)\) are uniformly Cauchy and so converge to a non-degenerate limiting metric \(\sigma_T\). By continuity at \(t = 0\), we can find \(\delta > 0\) such that
    \[\sup_B \eta (\sigma(t), \sigma(t')) < \epsilon\]
    for \(0 < t, t' < \delta\). The result then follows by the maximum principle for the heat equation applied to \(\eta\).
\end{proof}

\begin{lemma}
    \label{lem:flow-c1-bound}
    Suppose \(\sigma(t)\) are solutions to the family Hermite--Einstein flow on \([0, T)\), with
    \begin{enumerate}[(i)]
        \item \(\sigma(t) \to \sigma_T\) in \(C^0\), and
        \item \[\sup_B\abs{P(\sigma)}_{(h\sigma)_\mcF}\]
        uniformly bounded in \(t\). 
    \end{enumerate}
    Then \(\sigma(t)\) is bounded in \(C^1\) and in \(L^p_2\). Moreover, the curvature \(F_{h\sigma}\) and \(P(\sigma)\) are uniformly bounded in \(L^p\) for all \(p < \infty\).
\end{lemma}

\begin{proof}
    Suppose for contradiction that \(\sigma(t)\) was unbounded in \(C^1\). Then for some sequence \(t_i \to T\), with \(\sigma_i = \sigma(t_i)\),
    \[\sup_B \abs{\grad \sigma_i} = m_i\]
    is unbounded. By compactness, there exists \(b_i \in B\) such that
    \[\abs{\grad\sigma_i}(b_i) = m_i.\]
    Passing to a subsequence, we may assume that \(b_i \to b\). By working in a local trivialisation, we have reduced this to a purely local problem with local coordinates \(z_\alpha\) in the polydisc \(\{\abs{z_\alpha} < 1\}\), and with local sections \(\sigma_i(z)\). By translating the coordinates, we may assume that
    \[\sup_z\abs{\grad\sigma_i(z)} = m_i\]
    is attained at \(z = 0\).

    Next, we will rescale the coordinates. Define \(w_\alpha = m_i z_\alpha\), and then pullback \(\sigma_i\) via the map
    \[\{\abs{w_\alpha} < 1\} \to \{\abs{z_\alpha} < m_i^{-1}\},\]
    we obtain matrices \(\widetilde\sigma_i(w)\) defined on \(\abs{w_\alpha} < 1\), such that
    \[\sup_w\abs{\grad \widetilde\sigma_i} = 1\]
    is attained at \(w = 0\).

    We will now show that this contradicts our hypotheses. First, recall that
    \[F_{h\sigma} = F_h + \grad^{0, 1}(\sigma^{-1}\grad_h^{1, 0}\sigma),\]
    and so
    \[i\Lambda_H F_{h\sigma} = \sigma^{-1}\Delta_{H,h}^{1, 0}\sigma + Q(\sigma, \grad\sigma)\]
    where \(Q\) is (at most) quadratic in \(\grad\sigma\).

    By choosing a holomorphic local frame \(e_i\) for \(\mcF\), and writing \(\sigma = \sum_i \sigma^i e_i\), we obtain that
    \begin{equation}
        \label{eqn:expansion-laplacian-quadratic}
        P(\sigma) = \sigma^{-1}\sum_i(\Delta^{1, 0}_B \sigma^i)e_i + \widetilde Q(\sigma^i, \grad \sigma^i).
    \end{equation}
    We would like to apply this to the \(\widetilde\sigma_i\). By (ii), we know that \(P(\widetilde\sigma)\) is bounded, and by (i), we know \(\widetilde\sigma, \widetilde\sigma^{-1}\) are bounded. But we also have that \(\grad\widetilde\sigma_i\) is bounded, and so \(\widetilde Q(\widetilde \sigma, \grad\widetilde \sigma)\) is bounded as well. Thus, we obtain a bound in \(L^\infty\), and so \(L^p\), on
    \[\sum_i (\Delta_B^{1, 0}\widetilde\sigma^i)e_i\]
    over \(\{\abs{w_\alpha} < 1\}\), which is uniform on \(i\). Since \(\Delta_B^{1, 0}\) is elliptic, this means that \(\widetilde\sigma_i\) are bounded in \(L^p_2\). For \(p\) sufficiently large, the embedding \(L^p_2 \to C^1\) is compact, and so a subsequence of the \(\widetilde\sigma_i\) converges in \(C^1\), to some limit \(\widetilde\sigma\) say. By (i) and the definition of the \(\widetilde\sigma_i\), we know that \(\widetilde\sigma\) is in fact constant. On the other hand,
    \[\sup_w \abs{\grad\widetilde\sigma} = \lim_{i \to \infty}\abs{\grad\widetilde\sigma_i} = 1.\]
    Thus, we have a contradiction, and so the \(\sigma(t)\) are bounded in \(C^1\). Using the same estimates as before, we obtain that the \(\sigma(t)\) are bounded in \(L^p_2\). The curvature result follows immediately.
\end{proof}

\begin{theorem}
    \label{thm:hhe-flow-exists-for-all-time}
    The family Hermite--Einstein flow admits a unique smooth solution for all time.
\end{theorem}

\begin{proof}
    The family Hermite--Einstein flow may be written as
    \[\pdv{\sigma}{t} = L(\sigma) + Q(\sigma, \grad\sigma),\]
    where \(L\) is linear elliptic, and \(Q\) is a \emph{polynomial differential operator} of type \((2, 1)\), in the sense of Hamilton \cite{hamiltonHarmonicMapsManifolds1975}. Thus, as in \cite[Corollary 6.5]{simpsonConstructingVariationsHodge1988}, we may appeal directly to Hamilton's method in \cite[Section V.12]{hamiltonHarmonicMapsManifolds1975}, where we improve the \(L^p_2\)-estimate from \Cref{lem:flow-c1-bound} to an \(L^p_k\) estimate for all \(k \ge 2\).

    As Hamilton's method relies in improving an \(L^p_k\)-estimate to an \(L^p_{k+\delta}\) estimate, where \(0 < \delta < 1\), we need to define the Sobolev spaces \(L^p_k\) where \(k\) is a real number. This is done in \cite[Section 6.5]{kobayashiDifferentialGeometryComplex2014}.
    
    With all of this in mind, suppose we have a solution \(\sigma(t)\) on \([0, T)\). By what we have just discussed, there exists \(\sigma_T\) such that \(\sigma(t) \to \sigma_T\) in \(C^k\) for all \(k\). In particular, we may extend the flow to \([0, T]\). Short time existence then shows that a solution exists on \([0, T + \epsilon)\), and so a solution exists for all time.
\end{proof}

\subsection{Dirichlet problem for the family Hermite--Einstein equation}

\label{subsec:dirichlet-problem}

In \cite[Section 6]{simpsonConstructingVariationsHodge1988}, Simpson generalises the results of Donaldson in \cite{donaldsonSelfdualYangMillsConnections1985} to the case of manifolds with boundary, and considers the Dirichlet and Neumann problems for the Hermite--Einstein flow. The same modifications that Simpson makes apply in our case. To be more precise, let \((X, \omega_X)\) be a compact K\"ahler manifold, and let \((B, \omega_B)\) be a compact K\"ahler manifold with non-empty boundary \(\partial B\). Let \(\mcE \to X \times B\) be a smooth vector bundle, which is holomorphic over \(X \times \Int(B)\), and for each \(b \in \partial B\), is holomorphic over \(X \times \{b\}\). Let \(h\) be a Hermitian metric, such that for all \(b \in B\), \((\mcE_b, h)\) is Hermite--Einstein.

In particular, one needs to modify \Cref{lem:flow-c1-bound} to the case when \(B\) has boundary, where one must use boundary elliptic estimates as well as interior elliptic estimates. The same modifications as in \cite[Lemma 6.4]{simpsonConstructingVariationsHodge1988} apply, and so we may consider the Dirichlet problem for the family Hermite--Einstein flow as well.

Let \(\sigma_\partial\) be a smooth section of \(\End E\) over \(X \times \partial B\), such that for all \(b \in \partial B\), \(h\sigma_\partial\) is a Hermite--Einstein metric on \(\mcE_b\). Let \(\sigma_0\) be a smooth section of \(\mcF\) over \(B\), such that for all \(b \in B\), \(h\sigma_0\) is a Hermite--Einstein metric on \(\mcE_b\), and such that \(\sigma_0|_{\partial B} = \sigma_\partial\). The Dirichlet initial-boundary-value problem for the family Hermite--Einstein flow is
\begin{equation}
    \label{eqn:dirichlet-problem-family-Hermite--Einstein-flow}
    \begin{cases}
        \pdv{\sigma}{t} = -2\sigma(p_{h\sigma}(i\Lambda_H F_{h\sigma}) - i\lambda \nu_{h\sigma}(a)), & t > 0, b \in \Int(B), \\
        \sigma(0) = \sigma_0, & b \in B, \\
        \sigma(t) = \sigma_\partial, & b \in \partial B.
    \end{cases}
\end{equation}

Using the same proof as in \Cref{thm:hhe-flow-exists-for-all-time}, we obtain the following result.

\begin{theorem}
    The Dirichlet problem for the family Hermite--Einstein flow admits a unique solution.
\end{theorem}

In \cite{donaldsonBoundaryValueProblems1992}, Donaldson proves that the Dirichlet problem for the Hermite--Einstein equation,
\[\begin{cases}
    i\Lambda_\omega F_h = 0 & \text{on } \Int(B) \\
    h = h_\partial & \text{on } \partial B,
\end{cases}\]
admits a unique smooth solution.

\begin{proposition}
    \label{prop:heat-equation-decay}
    Let \(Z\) be a compact manifold with non-empty boundary \(\partial Z\). Suppose \(\theta \ge 0\) is a subsolution to the heat equation on \(Z \times [0, \infty)\). That is,
    \[\pdv{\theta}{t} + \Delta\theta \le 0.\]
    If \(\theta = 0\) on \(\partial Z\) for all \(t\), then \(\theta\) decays exponentially, with
    \[\sup_Z\theta(z, t) \le Ce^{-\mu t}.\]
\end{proposition}

This result follows from the spectral decomposition of the heat operator, and we refer to \cite[98]{donaldsonBoundaryValueProblems1992} for details. Donaldson's proof for the Dirichlet problem follows from applying \Cref{prop:heat-equation-decay} to \(\theta = \abs{i\Lambda_\omega F_h}^2\).

Using our results on the family Hermite--Einstein flow, we deduce the same result for the family Hermite--Einstein equation.

\begin{theorem}
    \label{thm:dirichlet-problem}
    Suppose \(\partial B \ne \emptyset\). The Dirichlet problem for the family Hermite--Einstein equation
    \[\begin{cases}
        p_h(i\Lambda_H F_h) - i\lambda \nu_h(a) = 0 & \text{on } X \times \Int(B) \\
        h = h_\partial & \text{on } X \times \partial B,
    \end{cases}\]
    admits a unique solution.
\end{theorem}

\begin{proof}
    Let
    \[\theta = \abs{P(\sigma)}^2\]
    be as before. In \Cref{prop:subsolution-to-heat-equation}, we prove that \(\theta\) is a subsolution to the heat equation. By the boundary condition, we know that \(\theta = 0\) on \(\partial B\) when \(t > 0\). Thus, by \Cref{prop:heat-equation-decay}, \(\theta\) decays exponentially, and so
    \[\int_0^\infty \sqrt{\theta} \dd t < \infty\]
    at each \(b \in B\). But \(\sqrt{\theta}\) is exactly the speed of the curve \(h\sigma(t)\) in the space of vertically Hermite--Einstein metrics. This space is complete, since over \(b \in B\) it is given as a homogeneous space \(G_b/K_b\). Thus, the \(\sigma(t)\) converges to some \(\sigma_\infty\). The Hermitian metric \(h\sigma_\infty\) satisfies the family Hermite--Einstein equation, as \(\theta \to 0\) as \(t \to \infty\).
\end{proof}

\bibliographystyle{alpha}
\bibliography{References.bib}

\end{document}